\documentclass[11pt,a4paper]{amsart}
\usepackage{amsmath,amsthm,amsfonts,amscd,amssymb,eucal,latexsym,mathrsfs}
\usepackage[all]{xy}
\usepackage{latexsym, amssymb, amscd, mathrsfs, graphics, graphicx, array,verbatim}%fullpage

\usepackage{MnSymbol}

\theoremstyle{plain}
\newtheorem{theorem}{Theorem}
\newtheorem{corollary}[theorem]{Corollary}
\newtheorem{lemma}[theorem]{Lemma}
\newtheorem{proposition}[theorem]{Proposition}

\theoremstyle{definition}
\newtheorem{remark}[theorem]{Remark}

\newtheorem{example}[theorem]{Example}

\newtheorem{question}[theorem]{Question}
\newtheorem{problem}[theorem]{Problem}
\newtheorem{definition}[theorem]{Definition}

\DeclareMathOperator{\Hom}{Hom}

\DeclareMathOperator{\coker}{coker}

\def\min{\mathop{\mathrm{min}}\nolimits}
\def\max{\mathop{\mathrm{max}}\nolimits}

\newcommand{\RI}{{\mathbb R}}\newcommand{\IE}{{\mathbb E}}\newcommand{\IR}{{\mathbb R}}
\newcommand{\ZI}{{\mathbb Z}}\newcommand{\IN}{{\mathbb N}}

\newcommand{\smatr}[4]{\left(\begin{smallmatrix} #1 & #2 \\ #3 & #4\end{smallmatrix}\right)}

\newcommand{\diag}[3]{\left(\begin{smallmatrix} #1 & 0&0 \\ 0&#2 & 0\\ 0&0& #3\end{smallmatrix}\right)}

\newcommand{\Id}{\mathrm{Id}}

\newcommand{\cF}{{\mathcal F}}

\newcommand{\sC}{{\mathscr C}}\newcommand{\sF}{{\mathscr F}}

\newcommand{\id}{{\rm Id}}

  \newcommand{\F}{\mathbf F} 
 
\newcommand{\G}{\Gamma}

\newcommand{\Hyp}{\mathrm{Hyp}}

\newcommand{\GL}{\mathrm {GL}}\newcommand{\PGL}{\mathrm {PGL}}
\newcommand{\SL}{\mathrm {SL}}\newcommand{\PSL}{\mathrm {PSL}}

\newcommand{\Sp}{\mathrm{Sp}}

\newcommand{\FI}{{\bf{F}}}

\newcommand{\e}{\varepsilon}

\newcommand{\val}{\mathrm{val}}

\newcommand{\impl}{\Rightarrow}
\newcommand{\ssi}{\Leftrightarrow}
\newcommand{\inj}{\hookrightarrow}
\newcommand{\surj}{\twoheadrightarrow}

\newcommand{\egdef}{:=}

\newcommand{\del}{\partial}

\newcommand{\IP}{\mathbb{P}}

\usepackage{tikz}

\newcommand{\acts}{\curvearrowright}

\newcommand{\Lk}{\mathrm{Lk}~}

\newcommand{\Gbf}{{\mathbf{G}}}

\newcommand{\Gal}{\mathrm{Gal}}

\newcommand{\virt}{\mathrm{virt}}

\date{\today}

\title{Random groups and nonarchimedean lattices} 

\author{Sylvain Barr\'e}
\address{\hskip-\parindent
Sylvain Barr\'e, Universit\'e de Bretagne Sud, Universit\'e Europ\'eenne de Bretagne, France}
\email{sylvain.barre@univ.ubs.fr}
\author{Mika\"el Pichot}
\address{\hskip-\parindent
Mika\"el Pichot, Dept. of Mathematics \& Statistics, McGill University, Montreal, Quebec, Canada H3A 2K6}
\email{pichot@math.mcgill.ca}

\begin{document}

\begin{abstract}
We consider  models of random groups in which the typical group is of intermediate rank (in particular, it is not hyperbolic). These models are parallel to M. Gromov's well-known constructions and include for example a ``density model" for groups of intermediate rank. The main novelty is the higher rank nature of the random groups. They are randomizations of certain families of lattices in algebraic groups (of rank 2) over local fields.
\end{abstract}

\maketitle

This paper introduces models of random groups ``of higher rank''. The construction, basic properties, and applications are detailed in \textsection \ref{S - basic idea} to  \textsection\ref{propertyT} below, which we now summarize.

The construction  (see \textsection \ref{S - basic idea}) is rather general. If $\G'$ is a group which acts properly on a simply connected complex $X'$ of dimension 2 with $X'/\G'$ compact, and $\G''\subset \G'$ is a subgroup of ``very large'' finite index, then one can  choose at random a family of  $\G''$-orbits of 2-cells $Y\subset X'$ inside $X'$. Then let $X$ denote the universal cover of $X'':=X'\setminus Y$.
  The random group $\G$ is the group of  transformations of the Galois covering 
\[
X\surj {X''}/{\G''}.
\]
This construction leads to several distinct  models of random groups including a ``density model", following M. Gromov. The initial structural data  ($\G'$, $X'$,...) for the model is called the \emph{deterministic data}. The basic properties of $\G$ 
 depend on the deterministic data.

An idea of groups ``of intermediate rank'' was introduced in \cite{rd} in particular to address the following question, where $X$ is CAT(0) and $X/\G$ is compact:

\begin{center}
$\IR^2\inj X\impl \ZI^2\inj \G$?
\end{center}
(This is the ``periodic flat plane problem" which has been formulated in many places, see \cite{gromov78} for an early reference.) Since the assumption $\IR^2\inj X$ is equivalent to $X$ being non hyperbolic, the new models are relevant to the study of this question.    We will see that in some cases (depending on the deterministic data, the density parameter, etc.) the answer is positive ``generically", but that the precise relation between the two conditions ``$\IR^2\inj X$'' and ``$\ZI^2\inj \G$'' remains mysterious even for random groups associated with lattices in $\PSL_3$. 

Before turning to these models let us discuss briefly Gromov's original construction of random groups and the density model  introduced in \cite{gromov0} (see also \cite[\textsection 6]{gromov78}, \cite{gromov81} or \cite{gromov1}).

Let $\G_1$ be a  hyperbolic group in the sense of Gromov, and  take successive quotients $\G_1\surj\G_2\surj \cdots$, say 
\[
\G_{n+1}:=\G_n/\langle\langle R_n\rangle\rangle
\] 
where $R_n\subset \G_n$ is a finite set   of additional relations. As  explained in \cite{gromov0} the set $R_n$  can  ``in general'' be chosen so that: 
\begin{itemize}
\item[(i)] $\G_{n+1}$ stays hyperbolic 
\item[(ii)] $\G_n\surj \G_{n+1}$  is injective on larger and larger balls 
\end{itemize} 
Property (ii) ensures the existence of an infinite limit group $\G_\infty$, while the attribute ``in general"   accounts for the oversupply of choices in the construction; its precise meaning depends on the size and the nature of $R_n$.  For example, if $\G_1$ is torsion free and the sets $R_n$ consist of  a single relation which is a ``higher and higher" power of the $n^\mathrm{th}$ element in a list exhausting $\G_1$, then the limit $\G_\infty$ is a finitely generated infinite torsion group  \cite[\textsection 4.5.C]{gromov0}. Here  (i) and (ii) become geometric assertions relying on $K<0$, and the construction offers almost total freedom. (Gromov's construction is related to the Burnside problem --- the existence of infinite torsion groups was established  by Golod, and the existence of  infinite groups of finite exponent by Adian and Novikov, and by Olshanskii using the small cancellation theory.)

A prominent feature is the \emph{genericity of hyperbolic groups}, as put forward in \cite{gromov0} and illustrated by (i) above.  M. Gromov has since invented  several models of random groups and constructed many exotic infinite groups using them \cite{gromov1, gromov00, gromov01,ollivier,ghys}. We are  interested here in his so-called \emph{density model}, which studies ``one-step" random quotients $\G_n\surj \G_{n+1}$ for ``very large" random sets $R_n$ of ``very long" relations.  If one starts with a free group $F_r$ on $r$ generators and let $\delta$ denote the density parameter, then the random group in the density model is a quotient of the form $F_r/\langle\langle W_p\rangle\rangle$, where $W_p$ is a set of $\approx|S_p|^{\delta}$ words chosen uniformly independently at random in the sphere $S_p$ of radius $p$ in $F_r$.  Gromov shows that if $\delta<1/2$, then the resulting random group is hyperbolic with overwhelming probability as $p\to \infty$, while  if $\delta>1/2$ it is trivial (meaning 1 or $\pm$) with overwhelming probability.  In this model, small cancellations occur for $\delta<1/12$. (One can also start here with a  non elementary hyperbolic group $\G$ and take random quotients by elements in the spheres $S_p\subset \G$, $p\to\infty$; the same phase transition ``$\delta<1/2\impl$ hyperbolic" and ``$\delta>1/2\impl 1/\pm$"  is then valid provided that $\G$ is torsion free \cite{ollivier-paper}.)  An earlier model of Gromov called  the ``few relator model" studies the situation where $|W_p|$ is bounded. We refer to \cite{ollivier} for a survey of these groups.

The groups of intermediate rank constructed in \cite{rd} can be put on a ``rank interpolation line'':
\begin{center} 
\includegraphics[width=8cm]{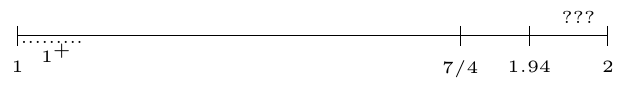}
\end{center}
 (we emphasis again that this only has a schematic value: as discussed in \cite{rd} already, the phenomenon of rank interpolation  is not unidimensional.) The two extreme cases in this picture are the hyperbolic groups (rank 1, or more generally the groups with isolated flats, of rank $1^+$) and the lattices in nonarchimedean groups (rank 2).
The value 1.94 refers to the bowtie group $\G_{\bowtie}$ introduced in \cite{rd} and further studied  in \cite{Haagerup}. The present paper constructs many groups whose rank is arbitrarily close to 2:  if the deterministic data arises from a nonarchimedean Lie group of rank 2, then the  ``rank'' of the random group (for example, the local rank in the sense of \cite[Definition 4.5]{chambers})   is as close to 2 as desired.  

We now formulate our main result in the special case of the density model  with deterministic data  the Cartwright--Steger lattices in $\PGL_3(\FI_q((y)))$ and their congruence subgroups. The techniques and constructions involved in the proof of this result apply in more general situations, and we will state and establish more general statements along the text. In fact, most of the assertions in Theorem \ref{intrdensity}, with the notable exception of the fact that ``$\delta<\frac 5 8\impl \ZI^2\inj \G$", will be proved under less restrictive assumptions.

The Cartwright--Steger lattices \cite{CS} are uniform lattices $\G_n<\PGL_n(\FI_q((y)))$ associated with the ring $R=\FI_q[y,1/y,1/(1+y)]\inj \FI_q(y)$. Their congruence subgroups $\G_n(I)$ correspond to ideals $I\triangleleft R$. The groups $\G_n$ act transitively on the vertices of the Bruhat--Tits building $X_n$ of $\PGL_n(\FI_q((y)))$. Below $n=3$; the random groups discussed in Theorem \ref{intrdensity} have as deterministic data the lattices $(X_3,\G_3,\{\G_3(I_p)\})$ associated with the Cartwright--Steger lattices of rank 2.

\begin{theorem} \label{intrdensity} 
Let $q$ be prime power. Fix two sequences $(f_p)_{p\geq 1}$ and $(s_p)_{p\geq 1}$ where $f_p\in\FI_q[y]$ is a monic irreducible polynomial prime to $y$ and $y+1$ and  $s_p\geq 1$ is an integer. 
If the density parameter $\delta$ satisfies 
\[
\delta<{5\over 8}
\] 
then the random group $\G$ in the density model of parameter $\delta$ with (Cartwright--Steger) deterministic data $(X_3,\G_3,\{\G_3(I_p)\})$ satisfies
\[
\ZI^2\inj \G
\]
with overwhelming probability, where the congruence subgroups $\G_3(I_p)_{p\geq 1}$ are associated with the ideal $I_p=\langle f_p^{s_p}\rangle$ generated by $f_p^{s_p}$ in $\FI_q[y,1/y,1/(1+y)]$. 
If in addition $s_p\geq k$ for $p$ large enough, then  
$\ZI^2\inj \G$ with overwhelming probability whenever
\[
\delta< {7k-3\over 7k+1}
\]
(which can be made as close to 1 as desired, independently of $q$). Furthermore, if 
\[
\delta<{q-1\over q-2}
\]
(which can be made as close to 1 as desired, independently of $k$) then $\G$ acts freely uniformly on a space $X$ of dimension 2 with the geodesic extension property, and if 
\[
\delta<{1\over 2}~\mathrm{~and~}~ q\geq 5
\]
then $\G$ has Kazhdan's property T with overwhelming probability. In addition, if $\delta_0$ is an arbitrary real number $<1$ given in advance, then there exists $q_0$ such that if $q\geq q_0$, then $\G$ has Kazhdan's property T with overwhelming probability for every $\delta<\delta_0$. If on the other hand  
\[
\delta>{q\over q-1}
\]
then $\G$ doesn't have property FA.  Finally, if 
\[
\delta<{1\over 2}
\]
(and $q$ is arbitrary), then $(X,\G)$ has the extension rigidity property of \cite{chambers} (namely, it ``remembers" the building $X_3$ it comes from). Finally, if $\delta_0<1$ and $\e>0$ are real numbers given in advance, then there exists $q_0$ such that if $q\geq q_0$, then $(X,\G)$ has  local rank (in the sense of \cite{chambers}) uniformly $\geq 2-\e$ with overwhelming probability for every $\delta<\delta_0$.
\end{theorem}

The paper is structured as follows. An analog of Gromov's few relator model is studied in \textsection \ref{S - pfp I} in relation with the periodic flat plane problem. The density model is studied from \textsection \ref{S - critical density} onwards, where we introduce ``critical densities" for various properties of groups in this model. In our framework, the density parameter $\delta$ regulates the size of the random subset $Y\subset X'$. For example, we have a critical density $\delta_T$ for  Kazhdan's property T, $\delta_{FA}$ for property FA, etc., and most importantly the critical density $\delta_{\ZI^2}$ for the property that $\ZI^2\inj \G$. The critical densities depends on the deterministic data. 
In \textsection \ref{S - Phase transition} we discuss various analogs of Gromov $\delta=1/2$ phase transition theorem in the density model, while \textsection \ref{sect:density} is devoted to estimating $\delta_{\ZI^2}$ for nonarchimedean lattices (in positive characteristic). Finally \textsection \ref{propertyT} derives other properties of the random group (for example property T) and its ``intermediate rank" behavior.  The reason why property T  arises only ``for sufficiently large residue fields" is the same as in Garland's paper \cite{Garland}, namely that the spectral gap is large enough only for $q$ large enough.

We conclude with a question, the issue of which seems hard to predict at this stage:

\begin{question}\label{Q-rank 2}
Is there $(X,\G)$ of local rank $>r$ such that $\ZI^2\not \inj \G$ for every $r<2$? 
\end{question}

Here $X$ is a CAT(0)  of dimension 2 and $\G\acts X$ freely with $X/\G$ compact and the local rank is defined in \cite{chambers}.  The question can be considered both in the general case, or when the order $q$ is  bounded. See \cite[Question 0.2]{chambers} for a related question (some aspects of the ``local to global problem" implicit in Question \ref{Q-rank 2} are also discussed in \cite{chambers}) and  also \cite[Problem 3]{Haagerup}.

\bigskip

\noindent{\bf Acknowledgements.} We are grateful to the referees for helpful comments on the text.

\tableofcontents

\section{Description of the random group}\label{S - basic idea}

Let $\G$ be a discrete group acting freely simplicially on a 2-complex $X$  with $X/\G$ compact, and  $\G_1, \G_2, \cdots$ be a family of finite index subgroups of $\G$ with $[\G:\G_p]\to \infty$, $p\to\infty$. 
The random group defined below is a  ``randomization" of the following (deterministic) sequence of covering maps:

\[
\xymatrix{
X\ar[dd]_{\G}\ar[rd]_{\G_1} \ar[rrd]_{\G_2}\ar[rrrd]_{\G_3}&&\\ 
& {X/\G_1}\ar[ld]\hspace{.1cm} & X/\G_2\ar[lld] \hspace{.3cm}& X/\G_3\ar[llld]\hspace{.3cm}\cdots\\
 {X/\G}}
\]
associated with $X$ and $(\G_p)_{p\geq 1}$. (The family may be nested $\G\geq \G_1\geq \G_2 \cdots$ and correspond to a tower $X/\G\leftarrow X/\G_1\leftarrow X/\G_2\leftarrow X/\G_3\leftarrow \cdots$ of compact spaces.) The randomization is achieved by inserting  a ``random topological noise" to the spaces $X/\G_p$ (where $p$ is very large) which is detected by the fundamental group. 

\begin{definition} We call $(X,\G,\{\G_p\})$ the \emph{deterministic data}.
\end{definition}

The topological perturbation is implemented as follows. 
 For each $p\geq 1$, remove a family of 2-cells in $X/\G_p$ (equivalently, a family of $\G_p$-orbits of 2-cells in the \emph{fixed} space $X$)  at random (with respect to a probability scheme for removing 2-cells, for example, Bernoulli). The universal cover $\tilde X_p$ of the resulting (random) space $X_p\subset X$ has a (random)  group $\tilde \G_p$ of deck transformations:
\[
\xymatrix{
 &\tilde X_p\ar[d]^{\tilde \G_p}&\\
X/\G_p  & X_p/\G_p\ar[l]}
\]
as described in the following overall diagram:
\[
\xymatrix{
&&& \tilde X_p\ar[llld]\ar[dd]^{\tilde \G_p}&\cdots\\
X\ar[dd]_{\G}\ar[rd]_{\G_1} \ar[rrd]^{\G_p}&&\\ 
& {X/\G_1}\ar[ld] \hspace{.5cm}\cdots&X/\G_p\ar[lld]  & X_p/\G_p\ar[l]\ar[llld]&\cdots\\
 {X/\G}}
\]
We study the properties of $(\tilde X_p,\tilde \G_p)$  when the order of approximation $p$ is very large. The construction provides a random group $\tilde \G_p$, a random space $\tilde X_p$, a random  action  $\tilde \G_p\acts \tilde X_p$,  and a random compact space $\tilde X_p/\tilde \G_p$.

At this stage, the model will be precisely defined as long as the random scheme removing 2-cells is specified; several options are possible. We study here  analogs of M.\ Gromov's well--known models for adding  relators at random to a  finitely generated group (more precisely, analogs of the ``few relator model" and  the ``density model" of  \cite{gromov0, gromov1}).

%Details for our variations on Gromov's models are provided in the next section below. 
%In Gromov's models, the random groups are  \emph{random quotients} of the group that is chosen to start with, while our groups are rather \emph{random extensions}.

Fix a deterministic data $(X,\G,\{\G_p\})$ (different choices for the triples $(X,\G,\{\G_p\})$ give rise to different models of random groups and lead to a priori distinct random objects). The set

\begin{center}
$\sC_p:=\{\G_p$-orbits of faces in $X\}$
\end{center}

\noindent (whose element are called \emph{equivariant faces} of $X$ with respect to $\G_p$, or sometimes \emph{equivariant chambers} when we have in mind a Bruhat--Tits building) is finite, and in many interesting cases it is rapidly growing. It
 plays the role of the set 
\begin{center}
$W_p:=$ all words (or reduced words)  of length $p$ (or at most $p$) in $\G$
\end{center}
 
 \noindent in Gromov's models, from which the relations are picked up at random and added to the given group $\G$ (e.g.\ the free group $\F_2$). 
 
For a finite subset $A=\{C_1,\ldots, C_k\}$ of $\sC_p$, we set
\begin{itemize}
\item $X_A:=X- \bigcup_{l=1}^k  \stackrel{\circ}{C_l}$
\item $K_A:=\pi_1(X_A)$ and $\tilde X_A\surj X_A$ be the corresponding universal cover
\item $\G_A$ be the Galois group of the covering map 
\[
\tilde X_A\surj X_A/\G_p.
\]
\end{itemize}

The random group is defined by:

\begin{definition}\label{def-lattice models} Fix for every $p\geq 1$ a process $\IP_p$ for selecting random subsets of elements in $\sC_p$. The \emph{random group of order $p$}  in the $(X,\G,\{\G_p\},\{\IP_p\})$-model is the group $\G_A$ associated by the construction above to the $\IP_p$-generic subset $A\subset \sC_p$.  We say that a property $P$ occurs \emph{with overwhelming probability} in this model if the probability that the random group $\G_A$ of order $p$ satisfies $P$ converges to 1 as $p\to\infty$.      
\end{definition}

\begin{remark}
1) In all cases considered below, the random process $\{\IP_p\}$ is \emph{universal}  in that it does not depend on the deterministic data $(X,\G,\{\G_p\})$. More precisely, a predetermined process $\IP$ is chosen for selecting a random finite subset in an arbitrary finite set, and this process $\IP$ is applied recursively to the terms of the sequence $(\sC_p)_p$.

2) Most of the results of the present paper extend easily to the case of \emph{proper} actions on \emph{cell} complexes with compact quotient. We also note that these models of random group only take into account the ``profinite information" contained in the deterministic data.  
\end{remark}

We  study two special cases of Definition \ref{def-lattice models}, the ``bounded model" and the ``density model". In the first model, a uniformly bounded number of elements of $\sC_p$ is chosen at random:

\begin{definition}[The bounded model]\label{fewmisschambers} Fix an integer parameter $c\geq 1$. The \emph{bounded model} over $(X,\G,\{\G_p\})$  is the $(X,\G,\{\G_p\},\{\IP_p\})$-model associated with the process
\begin{center}
$\IP_p:=$ ``choose $c$ chambers in $\sC_p$, uniformly and independently at random". 
\end{center}
\end{definition}

This corresponds to Gromov's ``few-relator model".

In the second model, the number of chosen chambers in $\sC_p$ is unbounded and rapidly growing.  

\begin{definition}[The density model]\label{fewmisschambers} Fix  a real parameter $\delta> 0$ (the density). The \emph{density model} over $(X,\G,\{\G_p\})$ is the $(X,\G,\{\G_p\},\{\IP_p\})$-model associated with the process
\begin{center}
 $\IP_p:=$ ``choose  $|\sC_p|^\delta$ chambers in $\sC_p$, uniformly and independently at random". 
\end{center}
\end{definition}

This corresponds to the Gromov ``density model".  As in Gromov models, the bounded model can be seen as a manifestation of the ``density model with $\delta=0$" (see also Section \ref{S - critical density}).

\begin{remark} In a sense, the new models can be thought of as  ``mirror images" of the Gromov models: rather than starting with a group with a large supply of quotients, for example non abelian free groups, and gradually adding relations at random, we typically start (see below) with lattices in some algebraic group of rank 2, which have ``as many relations as is conceivable" for an infinite group (in particular they are just infinite up to centre), and remove them at random. It is unclear how to randomize (say, residually finite) discrete groups of higher cohomological dimension;  for example, the above construction provides a precise meaning for the expression  ``$\G$ is a random extension of a lattice in $\PSL_n(K)$" for $n=3$, where $K$ is a local field --- what about $n>3$?
\end{remark}

\section{Preliminary results}\label{S - preliminary results}

Let $(X,\G,\{\G_p\})$ be a deterministic data, and let $p\geq 1$, $A\subset \sC_p$, $X_A$, $\G_A$,  $K_A$ (following the notation of  Section \ref{S - basic idea}) be fixed throughout the section. 

Observe that the sequence of covering spaces $\tilde X_A\surj X_A\surj X_A/\G_p$ provides an exact sequence (non-split in general) 
\[
1\to K_A \to \G_A\to \G_p\to 1
\] 
of discrete groups.

\begin{lemma} \label{L - cohomological dim} If $X$ is contractible, then:
\begin{enumerate}
\item The homology groups $H_i(X_A,\ZI)$ vanish for $i\geq 2$, and the covering space $\tilde X_A$ is contractible. 
\item The group $K_A$ is a free group on countably many generators. 
\item The group $\G_A$ is a finitely presented group of geometric dimension 2.  If $|A|\neq 0$, then  $\G_A$ is a strict extension of $\G$.
\end{enumerate}
\end{lemma}

\begin{proof} The first part of the first assertion is clear and the second part classically follows from the first ($\tilde X_A$ is weakly homotopy equivalent to a point, and therefore contractible). 
 The universal coefficient theorem 
\[
0\to \mathrm{Ext}(H_1(X_A,\ZI),\ZI)\to H^2(X_A,\ZI)\to \Hom(H_2(X_A,\ZI))\to 0,
\]
where the module $H_1(X_A,\ZI)$ is free and thus projective, shows that the second integral cohomology  $H^2(X_A,\ZI)$ vanishes.  Thus, $K_A$ is a free group. This follows from the Stallings--Swan theorem that a group of cohomological dimension 1 is free.  
(Note that the group $K_A$ does not act ``naturally" on a tree in general.) Let us prove (3). 
 Since $\G$ acts freely on $X$ contractible, $X/\G$ is an Eilenberg--MacLane space, $K(\G,1)$, and it follows from (1) that $\tilde X_A/\G_A$ is a $K(\G_A,1)$  and in particular $\G_A$ is finitely presented of geometric dimension 2. (If the action $\G\acts X$ is only assumed to be proper, then the groups $\G_A$ are of proper (Bredon) geometric dimension 2.) 
Assume that $|A|\neq 0$. Let $\gamma$ be the boundary of an element of $A$. If $\G_A\to \G$ is an isomorphism, then $\gamma$ is a boundary in $X_A$, and therefore is an equator in a 2-sphere of $X$. This contradicts the fact that $X_A$ is aspherical, so  $K_A\neq 1$, and by equivariance, $K_A$ is infinitely generated. 
\end{proof}

\begin{remark}\label{R - 10}
If $\G$ is of higher cohomological dimension, then $K_A$ is not necessarily free. The density model is interesting to study in this situation, and seems to be working properly only in ``high density regimes" (this to compensate, especially in the case of $\SL_n$, $n\geq 4$, for the very strong ambient rigidity properties of the deterministic data).
\end{remark}

The following ``flat plane correspondence" assertion is  useful in connection with Gromov's periodic flat plane problem.

\begin{lemma}\label{L - lifting lemma}
Assume that $X$ is a CAT(0) space and that $\G\acts X$ is isometric.  For every flat $\Pi\subset X_A$, there exists a flat $\tilde \Pi\subset \tilde X_A$ such that the restriction of $\tilde \pi_A : \tilde X_A\to X_A$ to $\tilde \Pi$ is an isometry onto $\Pi$. Furthermore, if there exists  a subgroup $\Lambda\subset\G$,  $\Lambda\simeq \ZI^2$, such that $\Pi/\Lambda$ is a compact torus, then there is a corresponding  a subgroup $\tilde \Lambda\subset\G_A$,  $\tilde \Lambda\simeq \ZI^2$, such that $\tilde \Pi/\tilde \Lambda$ is a compact torus,  and the following diagram commutes:
\begin{center}
\hskip1cm 
$\xymatrix{
 &\mbox{~}\mbox{~}\tilde \Lambda\subset \G_A\ar[d]&\\
\ZI^2\ar[r]\ar[ur]  & \Lambda\subset \G_p}$
\end{center}
Conversely, if $\tilde \Pi\subset \tilde X_A$ is a flat in $\tilde X_A$, then $\tilde \pi_A(\tilde \Pi)$ is a flat in $X_A$. If in addition  there exists a subgroup $\tilde \Lambda\subset\G_A$,  $\tilde \Lambda\simeq \ZI^2$, such that $\tilde \Pi/\tilde \Lambda$ is a compact torus, then its projection $\Lambda$ in $\G$ is a subgroup isomorphic to $\ZI^2$, such that $\Pi/\Lambda$ is a compact torus.
\end{lemma}
\begin{proof}
Since $\pi_A: \tilde X_A\to X_A$ is a locally isometric covering map and $\IR^2$ is contractible, the map $j:\RI^2\to \Pi\subset X_A$ admits a unique isometric lifting  $\tilde j :   \RI^2\to\tilde \Pi\subset \tilde X_A$ through any point $\tilde j(0)=\tilde x\in \tilde X_A$ such that $\pi_A(\tilde x)=j(0)=:x$, such that the following diagram commute:
\[
\xymatrix{
 &\tilde X_A\ar[d]^{\pi_A}&\\
\RI^2\ar[r]_{j}\ar[ur]^{\tilde j}  & X_A}
\]
The map $\tilde \pi_A$, being a local isometry,  restricts to an isometry from $\tilde \Pi$ onto $\Pi$. 
Let $\Lambda\subset \G$ be a subgroup isomorphic to $\ZI^2$ such that $\Pi/\Lambda$ is a compact torus. Let $s\in \Lambda$, and let $g_s\in  \G_A$ be any lift of  $s\in \Lambda$. By equivariance of $\pi_A$, we have $\pi_A(g_s(\tilde x))=sx$. Choose $k_s\in K_A$ such that $k_sg_s\tilde x\in \tilde \Pi$. Then $(k_sg_s)^{-1}(\tilde \Pi)$ is an isometric lifting of $\Pi$ which contains $\tilde x$.  Uniqueness of liftings implies that $k_sg_s(\tilde \Pi)=\tilde \Pi$. Let $\{a,b\}$ be a generating set of $\Lambda$ and let $\tilde a,\tilde b\in \G_A$ be such that $\tilde a(\tilde \Pi)=\tilde \Pi$ and $\tilde b(\tilde \Pi)=\tilde \Pi$. Since $\pi_A(\tilde a\tilde b\tilde x)=ab\pi_A(\tilde x)=ba\pi_A(\tilde x)=\pi_A(\tilde b\tilde a\tilde x)$ and $\pi_A : \tilde \Pi\to \Pi$ is isometric, the elements $\tilde a$ and $\tilde b$ of $\G_A$ commute  and therefore generate an (infinite torsion-free) abelian group $\tilde \Lambda\subset \G_A$. If $\Lambda\simeq \ZI$, let $\tilde c$ be a generator, and let $c\in \G_p$ be its image. Then $c(\Pi)=\Pi$ and $\Lambda\subset \langle c\rangle$,  contradicting $\Lambda\simeq \ZI^2$. Therefore $\tilde \Lambda\acts \tilde\Pi$ is conjugate to a discrete action of $\ZI^2$ on $\IR^2$, hence $\tilde\Pi/\tilde \Lambda$ is compact.   

Conversely, let $\Pi$ be the image of $\tilde \Pi$ under $\pi_A$. Since ${\pi_A}_{|\tilde \Pi}$ is a locally isometric covering map, $\Pi$ is either a torus, a cylinder, or a flat plane.  The first two possibilities are incompatible with the fact that $\Pi\subset X_A\subset X$, where $X$ is a CAT(0) space of dimension 2. In particular, we see that if $\tilde \Lambda =\langle \tilde a,\tilde b\rangle$ is a subgroup isomorphic to $\ZI^2$ acting on $\tilde \Pi$ with $\tilde \Pi/\tilde \Lambda$ compact, then $K_A\cap \tilde \Lambda=\{e\}$.  But this implies that its image $\Lambda$ in $\G$ is isomorphic to $\ZI^2$, acting freely on $\Pi$ with $\Pi/\Lambda$ compact. 
\end{proof}

 \begin{lemma}\label{L - hyperbolic} Assume that $X$ is a CAT(0) space and that $\G$ acts isometrically. Assume that $A\subset B\subset \sC_p$. 
 \begin{enumerate}
\item If $\G_A$ is  Gromov hyperbolic,  then $\G_B$ is Gromov hyperbolic. 
 \item If $\G_B$ contains a subgroup isomorphic to $\ZI^2$,  then $\G_A$ contains a subgroup isomorphic to $\ZI^2$.
 \end{enumerate}
\end{lemma}
 
 \begin{proof}
(1) By invariance under quasiisometry,  it is enough to prove that $\tilde X_B$ is hyperbolic. If it is not hyperbolic then, being a CAT(0) space,  it contains a flat plane $\Pi\simeq \RI^2\inj \tilde X_B$ by the flat plane theorem. Since the  image of $\Pi$ in $X_B$ under the covering map $\pi_B\colon \tilde X_B\surj X_B$ is a flat plane, the space  $X_A$ contains a flat plane, and therefore, $\tilde X_A$ contains a flat plane by Lemma \ref{L - lifting lemma}. Therefore, $\tilde X_A$ is not hyperbolic and $\G_A$ is not a hyperbolic group, contrary to assumption. Thus $\G_B$ is hyperbolic.
 
 (2) Since $A\subset B$, the map $\tilde X_B\surj X_B\inj X_A$ lifts to a map  
 \[
 \tilde X_B\to \tilde X_A\surj X_A
 \]  
 Let $Y=\tilde X_A$, $\Delta=\G_A$, and let $C$ be the pull-back of $B\setminus A\subset X$ in $Y$. Let also $\tilde \Lambda\subset \G_B$ be a subgroup isomorphic to $\ZI^2$, and let $\tilde \Pi\subset \tilde X_B$ be the corresponding periodic flat.   We apply Lemma \ref{L - lifting lemma} to the isometric action $\Delta\acts Y$, and the subset $C$ of $Y$. Then $Y_C$ coincide with the image of $\tilde X_B$ in $Y$ and $\tilde Y_C=\tilde X_B$, $\Delta_C=\G_B$. It follows that $\Pi=\tilde \pi_C(\tilde \Pi)$ is a flat in $Y_C\subset Y=\tilde X_A$ and that the projection $\Lambda$ of $\tilde \Lambda$ in $\Delta$ is a subgroup isomorphic to $\ZI^2$, such that $\Pi/\Lambda$ is compact. This shows that $\G_A$ contains a copy of $\ZI^2$.
 \end{proof}
 
 The same is true for the isolated flats property:

\begin{lemma}\label{L - isol flats} Assume that $X$ is a CAT(0) space and that $\G$ acts isometrically. If $X$ has the isolated flats property,  then $\tilde X_A$ has the isolated flats property. 
 \end{lemma}

 \section{Random periodic flat plane problems I}\label{S - pfp I}

 We consider the bounded model first.

Recall that a countable group $\G$ is called  \emph{virtually indicable} if some finite index subgroup of $\G$ admits an infinite abelian quotient.
 
The main theorem in this section is a result that is significantly more general than Theorem \ref{intrdensity} but is  restricted to the bounded model. It produces groups which are \emph{infinitesimal perturbations} of the deterministic data.

\begin{lemma}\label{L -notvirtind} Assume that the deterministic data $(X,\G,\{\G_p\})$ satisfies:
\begin{itemize}
\item $X$ is a CAT(0) space and $\G$ acts isometrically
\item $\G_p\triangleleft \G$, and $[\G:\G_p] \to_p \infty$
\item $\G$ is not virtually indicable
\item $\ZI^2\inj \G$
\end{itemize} 
then the random group in the bounded $(X,\G,\{\G_p\})$-model   contains a copy of $\ZI^2$ with overwhelming probability.
\end{lemma}

\begin{proof} 
Let $\Lambda$ be a subgroup of $\G$ isomorphic to $\ZI^2$.
 Let $\Pi\inj X$ be a flat plane on which $\Lambda$ acts freely with $X/\Lambda$ a compact torus (which exists by the flat torus theorem).  Consider the subgroup 
 \[
 \Lambda_p=\G_p\cap \Lambda
 \] 
 and let $F_p$ be the set of $\Lambda_{p}$-orbits of chambers in $\Pi$.  Since 
 \[
 [\Lambda:\Lambda_p]\leq [\G:\G_p]
 \]
 the set $F_p$ is finite. We denote by $F_p'\subset \sC_p$  the image of $F_p$ into $\sC_p$ under the map which associates to a $\Lambda_p$-orbit in $\Pi$  its corresponding $\G_p$-orbit in $X$.
  
Let $(\G_{A_p},\tilde X_{A_p})$ be the random group associated with a random subset $A_p\subset \sC_p$ in the bounded model (so $|A_p|=c\geq 1$ is a fixed integer). We will show that  
\begin{center}
$\IP(A_p\cap \Pi\neq\emptyset)\xrightarrow[p\to\infty]{} 0$. 
\end{center}
Assume towards a constradiction that the limit is not zero. Then there exist $\delta>0$ and a sequence $p_1<p_2<p_3<\ldots$ of integers such that, for all $i\geq 1$, the random set $A_{p_i}$ in  $\sC_{p_i}$ contains, with probability at least $\delta>0$, an equivariant chamber $C$ such that $C\cap \Pi\neq \emptyset$:
\[
\IP(\exists C\in A_{p_i}\mid C\cap \Pi\neq \emptyset)  \geq \delta.
\]
 A fortiori,
\[
\sum_{C\in A_{p_i}} \IP(C\cap \Pi\neq \emptyset)  \geq \delta.
\]
Thus, for all $i\geq 1$, 
\[
|\{C\in \sC_{p_i}| C\cap \Pi\neq \emptyset\}| \geq {\delta\over c} |\sC_{p_i}|.
\]
Since the condition $C\cap \Pi\neq \emptyset$ is equivalent to the fact that $C\in F_{p_i}'$,  we obtain
\[
|F_{p_i}|\geq|F_{p_i}'|\geq |\{C\in \sC_{p_i}| C\cap \Pi\neq \emptyset\}| \geq {\delta\over c} |\sC_{p_i}|.
\]
Consider the groups $G_p=\G/\G_p$ and $H_p=\Lambda/\Lambda_p$, which act freely on $\sC_p$ and $F_p$ respectively. 

It follows that there exists a constant $\kappa$, which depends only on $\delta$, $c$, the number of $\G$-orbit of chambers in $X$ and the number of $\Lambda$-orbits of chambers in $\Pi$, such that 
\[
[G_{p_i}:H_{p_i}]\leq \kappa.
\]
Hence, for every $i$, the subgroup $\Lambda\G_{p_i}$ of $\G$ is  (by the isomorphism theorems) 
of index
\[
[\G:\Lambda\G_{p_i}]\leq \kappa.
\]
As $\G$ is finitely generated (for it acts freely on $X$ with $X/\G$ compact), the family $\cF$ of subgroups of index at most $\kappa$ is finite.  

A finitely generated group $\G$ is not virtually indicable  if and only if every finite index subgroup $\G_0$ of $\G$ has finite  abelianization $\G_0/[\G_0,\G_0]$. Since  $\Lambda\G_{p_i}\in \cF$ we have  
\[
|\Lambda\G_{p_i}/[\Lambda\G_{p_i},\Lambda\G_{p_i}]|<\gamma
\] 
for some fixed $\gamma>0$, not depending on $i$.   On the other hand, since $\Lambda\simeq \ZI^2$,
\[
H_{p_i}\subset \Lambda\G_{p_i}/[\Lambda\G_{p_i},\Lambda\G_{p_i}]
\]
and as
\[
|H_{p_i}|\xrightarrow[|G_{p_i}|\to \infty]{} \infty
\]
we have
\[
|\Lambda\G_{p_i}/[\Lambda\G_{p_i},\Lambda\G_{p_i}]|\xrightarrow[p_i\to \infty]{} \infty.
\] 
which gives the desired contradiction. 
Thus, 
\[
\IP(A_p\cap \Pi\neq\emptyset)\xrightarrow[p\to\infty]{} 0.
\] 
We now apply Lemma \ref{L - lifting lemma}, whose assumptions are satisfied with overwhelming probability. This shows that $\ZI^2\inj \G_{A_p}$ with overwhelming probability.
\end{proof}

Theorem \ref{T- stable random} below can be expressed by saying that ``the periodic flat plane problem is stable under randomization in the bounded model", under suitable assumptions. More precisely:

 \begin{definition} We say that a property $P$ is \emph{stable under randomization} in the $(X,\G,\{\G_p\},\{\IP_p\})$-model if the random group has property $P$ with overwhelming probability provided that the deterministic data $(X,\G,\{\G_p\})$ has property $P$.
\end{definition}

 \begin{theorem}\label{T- stable random} Let $(X,\G,\{\G_p\})$ be a deterministic data, where $\G$ acts isometrically on $X$ CAT(0), and $\G_p\triangleleft \G$ is a sequence of normal subgroups such that $[\G:\G_p]\to \infty$. If $\G$ is not virtually indicable, then the periodic flat plane alternative is stable under randomization in the bounded  $(X,\G,\{\G_p\})$-model. 
 \end{theorem}

\begin{proof} Assume that $\G$ satisfies the periodic flat plane alternative and let us prove that the random group does.  
If $\G$ contains $\ZI^2$, then Lemma \ref{L -notvirtind} shows more: every copy of $\ZI^2$ eventually survives in the random group. If $\G$ does not contain $\ZI^2$ then it is hyperbolic, and Lemma  \ref{L - hyperbolic} shows that the random group is hyperbolic as well. 
\end{proof}

According to the ``flat closing conjecture" ($=$ the hypothesis that the periodic flat plane problem always has a positive answer), one expects that virtual indicability doesn't play a role regarding the existence of a periodic flat plane in the bounded model. The conjecture implies that for an arbitrary deterministic data $(X,\G,\{\G_p\})$, where $X$ is CAT(0) and $\G$ acts isometrically, the periodic flat plane alternative is stable under randomization in the bounded $(X,\G,\{\G_p\})$-model. This can be checked easily in a few situations, including spaces with isolated flats, using Lemma \ref{L - isol flats}.

%Obviously, the existence of PFP itself is not a property stable under randomization in the bounded model in general -- and a fortiori in the density model -- choose for example $\G=\ZI^2$, $X=\IR^2$ and $\G_p=(p\ZI)^2$.)  

Let us conclude this section with a statement which (in the case of the bounded model) weakens the assumption on the deterministic data for  the $\ZI^2\inj \G$ assertion in Theorem \ref{intrdensity}.
 
  Let $k$ be a non-archimedean local field with discrete valuation, and $\mathbf{G}$ be an algebraic group of rank 2 over $k$.
 Let  $X$ be the associated Bruhat-Tits building,  $\G$ be a uniform lattice in $\mathbf{G}$ acting freely on $X$, and $\{\G_p\}_p$ be a sequence of finite index normal subgroups, and let us take  $(X,\G,\{\G_p\})$ as deterministic data.

\begin{corollary}\label{C- random building with chambers missing bounded model}
If $\G$ is the random group in the bounded $(X,\G,\{\G_p\})$-model, then $\ZI^2\inj \G$ with overwhelming probability.
\end{corollary}

Note that such a determinitic data can be constructed. Indeed, let $\G'$ be any uniform lattice in $\mathbf{G}$ (see \cite[Chap. IX.3]{Margulis}; if the characteristic of $k$ is zero, then all lattices are all uniform). Then, by a well-known result of Selberg,   $\G$ has a torsion-free subgroup of finite index $\G$, and since $\mathbf{G}$  over $k$ is a linear group,   Malcev's theorem shows that $\G$  residually finite. Now $\G_p$ may be taken to be any sequence of finite index subgroups of $\G$, for example a sequence of normal subgroups with trivial intersection.

\begin{proof}
 Being an algebraic group of rank 2, $\mathbf{G}$ satisfies Kazhdan's property T (see \cite[Chap. 1, Theorem 1.6.1]{HV}) and so does the lattice $\G$. A fortiori, $\G$ is not virtually indicable and  therefore Lemma \ref{L -notvirtind} applies. 
\end{proof}

\section{Critical densities, factors and lifts}\label{S - critical density}

The present section defines two relevant critical densities for studying the periodic flat plane problem in the density model:   $\delta_{\IR^2}$ (the critical  density for being hyperbolic) and $\delta_{\ZI^2}$ (the critical density for containing $\ZI^2$).  Existence follows from the ``lift/quotient stability in the density model'' of the corresponding group property (compare Proposition \ref{P - hyp-Z2 critical densities}).  

We call a group property $P$ \emph{lift stable} (resp.\ \emph{factor stable} or \emph{quotient stable}), if for every surjective morphism $\G\surj \G'$ where $\G'$ (resp.\ $\G$)  has property $P$, then $\G$ (resp.\ $\G'$) has property P.   It is clear that
\begin{itemize}
\item 
 $P$ lift stable $\ssi$  $\neg P$ factor stable
 \item $P$ lift stable \&  $P$ factor stable $\impl$ $P$ trivial. (In this respect factor and lift stable properties are  ``half-trivial'' properties.)
\end{itemize}

The following result explains the origin, for lift/factor stable group properties, of transition thresholds observed in the density models. 

\begin{proposition}\label{P -Phase transition}
Let $P$ be a lift stable (resp.\ $P'$ be a factor stable) group property and let $(\G,X,\{\G_p\})$ be a deterministic data. There exists a critical parameter $\delta_{P}=\delta_{P}(\G,X,\{\G_p\})$ (resp.\ $\delta_{P'}=\delta_{P'}(\G,X,\{\G_p\})$) such that the random group $\G$ in the density model of parameter $\delta$ over $(\G,X,\{\G_p\})$ satisfies the following:
\begin{itemize}
\item $\G$ has property $P$   if $\delta>\delta_P$ (resp. $\G$ has property $P'$ if $\delta<\delta_{P'}$) 
\item $\G$ doesn't have property $P$  if $\delta<\delta_P$ (resp. $\G$ doesn't have property $P'$ if $\delta>\delta_P'$) 
\end{itemize}
 (with overwhelming probability). In particular, $\delta_{\neg P}\leq \delta_{P}$ and  $\delta_{P'}\leq \delta_{\neg P'}$ respectively  --- leading to a ``threshold'':
 \[
\includegraphics[width=4cm]{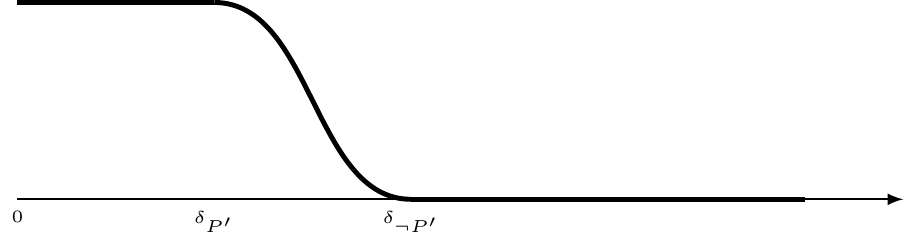}
\]
\begin{center}\emph{``Half-trivial'' properties have one threshold} 
\end{center}
\end{proposition}

It seems that in known cases $\delta_P=\delta_{\neg P}$. For  $P$ such that  $\delta_P>\delta_{\neg P}$ it could be informative to determine the true behaviour of the probabilities on the interval between $\delta_{\neg P}$ and $\delta_P$.

\begin{example}
The proposition shows the existence of critical densities for
\begin{itemize}
\item Kazhdan's property T (factor stable), say $\delta_T$,
\item  Serre's property FA (factor stable), say $\delta_{FA}$ 
\item largeness, which is the existence of a finite index subgroup of $\G$ surjecting to a non abelian free group (lift stable), say $\delta_{\surj F_2}^{\virt}$
\item virtual indicability (lift stable), say $\delta_{\surj\ZI}^{\virt}$
\end{itemize}
where 
\[
\delta_T\leq \delta_{FA}\leq \delta^{\virt}_{\surj F_2}\leq \delta_{\surj \ZI}^{\virt}
\]
 Observe that if  $P\impl P'$ where both $P,P'$ are factor stable/both are lift stable or if $P\impl \neg P'$ where $P$ is factor stable and $P'$ is lift stable, then $\delta_P\leq \delta_{P'}$.
\end{example}

The proof of Proposition \ref{P -Phase transition} will follow from an assertion that ``the random group at density  $\delta'$ is a quotient of the random group at a higher density $\delta>\delta'$", which we make precise in Lemma \ref{L -Phase transition}. The existence of transition thresholds does not generalize to properties which are not lift stable or factor stable. For example, a property like $T\vee  \stackrel{\virt}{\surj} \ZI$  admits a priori two thresholds. In general, the above result can be extended to any finite conjunction of properties which are lift stable or factor stable, resulting in a finite number of phase transitions.

We prove Lemma  \ref{L -Phase transition} in a more general model, called the ``$\alpha$-model", which extends both the bounded model (where $\alpha$ is bounded) and the density model of parameter $\delta$ for any $\delta>0$ (where $\alpha_p=|\sC_p|^\delta$):

\begin{definition}  Let $(X,\G,\{\G_p\})$ be a deterministic data and $\alpha: \IN\to\IN$ be an increasing sequence.   The \emph{$\alpha$-model} over $(X,\G,\{\G_p\})$ is the $(X,\G,\{\G_p\},\{\IP_p\})$-model of random groups in which $\IP_p$ selects $\alpha_p$ chambers in $\sC_p$, uniformly and independently at random. Given $\alpha,\beta : \IN\to\IN$, we say that the $\alpha$-model is dominated by the $\beta$-model if $\alpha_p\leq \beta_p$ for every $n$ sufficiently large.
\end{definition}

Zero density models, in which  say $\alpha_p:=o(|\sC_p|^\delta)$ for every $\delta>0$ (this includes the bounded model), can be used to provide variations on the model density $\delta>0$, for example: $\alpha_p:=|\sC_p|^\delta + o(|\sC_p|^\delta)$.

The bounded model is dominated by the density model of parameter $\delta$, which is dominated by the density model of parameter $\delta'>\delta$.
Lemma \ref{L -Phase transition} makes explicit the fact that if $\alpha$ dominates $\beta$, then ``the random group in the $\alpha$-density model is a quotient of the  random group in the $\beta$-density model".
We prove the case of lift stable properties only, since the corresponding statement for factor stable properties is  analogous. 

\begin{lemma}\label{L -Phase transition}
Let $P$ be a lift stable group property and  $(\G,X,\{\G_p\})$ be a  deterministic data. If the $\alpha$-model on $(\G,X,\{\G_p\})$ is dominated by the $\beta$-model on $(\G,X,\{\G_p\})$, then the probability that the random group has property P in the $\alpha$-model is dominated by the probability that the random group has property $P$ in the $\beta$-model. 
\end{lemma}

In particular, if $P$ occurs with overwhelming probability in the $\alpha$-model, then it does so in the $\beta$-model, establishing Prop.\ \ref{P -Phase transition}.  

\begin{proof}
If $P$ is lift stable, then 
\begin{align*}
|\{c_1,\ldots,c_{\alpha_p}&\in \sC_p\mid \G_{\{c_1,\ldots,c_{\alpha_p}\}}\mathrm{~has~property~P}\}|\\
&= 
{{|\{c_1,\ldots,c_{\beta_p}\in \sC_p\mid \G_{\{c_1,\ldots,c_{\alpha_p}\}}\mathrm{~has~property~P}\}|}\over {|\sC_p|^{\beta_p-\alpha_p}}}\\
&\leq 
{{|\{c_1,\ldots,c_{\beta_p}\in \sC_p\mid \G_{\{c_1,\ldots,c_{\beta_p}\}}\mathrm{~has~property~P}\}|}\over {|\sC_p|^{\beta_p-\alpha_p}}}
\end{align*}
Therefore if $A_\alpha$ denote the random subset of $\sC_p$ in the $\alpha$-model and $A_\beta$ denote the random subset of $\sC_p$ in the $\beta$-model, then:
\begin{align*}
\IP(\G_{A_\alpha}\mathrm{~has~property~P})\leq \IP(\G_{A_\beta}\mathrm{~has~property~P}).
\end{align*}
This proves the assertion.
\end{proof}

However, a property like ``hyperbolicity'' is not factor stable nor lift stable, even though Gromov establishes sharp phase transition in his density model.

The following weakening of lift/factor-stability is consistent with the density model and allows for more general properties.

\begin{definition}\label{D-Liftquotient stable density} 
A group property $P$ is said to be \emph{lift stable (resp.\ factor stable or quotient stable) in the density model} if the random group at density $\delta$ satisfies property $P$ with overwhelming probability  whenever the random group at density $\delta'<\delta$ (resp.\ $\delta'>\delta$) satisfies property $P$ with overwhelming probability.   
\end{definition}

Thus, we have (tautologically) a critical density $\delta_{P}(\G,X,\{\G_p\})$ for every lift stable property $P$ in the density model, and similarly for factor stable properties. 
Furthermore
\[
\delta_{\neg P}(\G,X,\{\G_p\}) \leq \delta_{P}(\G,X,\{\G_p\})
\]
if $P$ is lift stable and $\neg P$ factor stable.

\begin{proposition}\label{P - hyp-Z2 critical densities} Let $(\G,X,\{\G_p\})$ be a deterministic data, where $X$ is a CAT(0) space and $\G$ acts isometrically. 
Then:
\begin{enumerate}
\item  Gromov-hyperbolicity is lift stable in the density model over $(\G,X,\{\G_p\})$. Therefore, there exists a critical density, denoted $\delta_{\Hyp}(\G,X,\{\G_p\})$, for the random group in the density model to be hyperbolic, and a critical density, denoted $\delta_{\IR^2}(\G,X,\{\G_p\})$, for the random universal cover to contain $\IR^2$.
\item The existence of  subgroups isomorphic to $\ZI^2$ is factor stable in the density model.  Therefore, there exists a critical density, denoted $\delta_{\ZI^2}(\G,X,\{\G_p\})$, for the random group in the density model to contain a subgroup isomorphic to $\ZI^2$, and a critical density, denoted $\delta_{\neg \ZI^2}(\G,X,\{\G_p\})$, for the random group to contain no subgroup isomorphic to $\ZI^2$.
\end{enumerate}
\end{proposition}

By the periodic flat plane theorem, 
\[
\delta_{\IR^2}(\G,X,\{\G_p\}) \geq \delta_{\ZI^2}(\G,X,\{\G_p\}).
\]

\begin{problem}
Given a deterministic data $(\G,X,\{\G_p\})$ where $X$ is a CAT(0) space and $\G$ acts properly uniformly isometrically, is it true that 
\[
\delta_{\IR^2}(\G,X,\{\G_p\})=\delta_{\ZI^2}(\G,X,\{\G_p\})?
\]
\end{problem}

The flat closing conjecture implies that the answer is positive.

\begin{proof}[Proof of Proposition \ref{P - hyp-Z2 critical densities}]
(1) By Lemma \ref{L - hyperbolic} (a), we have:
\begin{align*}
|\{c_1,\ldots,c_{\alpha_p}&\in \sC_p\mid \G_{\{c_1,\ldots,c_{\alpha_p}\}}\mathrm{~is~hyperbolic}\}|\\
&= 
{{|\{c_1,\ldots,c_{\beta_p}\in \sC_p\mid \G_{\{c_1,\ldots,c_{\alpha_p}\}}\mathrm{~is~hyperbolic}\}|}\over {|\sC_p|^{\beta_p-\alpha_p}}}\\
&\leq 
{{|\{c_1,\ldots,c_{\beta_p}\in \sC_p\mid \G_{\{c_1,\ldots,c_{\beta_p}\}}\mathrm{~is~hyperbolic}\}|}\over {|\sC_p|^{\beta_p-\alpha_p}}}
\end{align*}
This shows that the random group in the $\beta$-model is hyperbolic provided that the random in the $\alpha$-model is hyperbolic, for any $\beta$ dominating $\alpha$, which applies in particular to the density model. This shows the existence of $\delta_{\Hyp}$ and similarly of $\delta_{\IR^2}$. %replacing ``$\leq$'' with ``$\geq$''.

(2) The same proof applies (using now Lemma \ref{L - hyperbolic} (b)), with a reverse inequality. 
Namely we have:
\begin{align*}
|\{c_1,\ldots,c_{\alpha_p}&\in \sC_p\mid \G_{\{c_1,\ldots,c_{\alpha_p}\}}\mathrm{~contains~}\ZI^2\}|\\
&= 
{{|\{c_1,\ldots,c_{\beta_p}\in \sC_p\mid \G_{\{c_1,\ldots,c_{\alpha_p}\}}\mathrm{~contains~}\ZI^2\}|}\over {|\sC_p|^{\beta_p-\alpha_p}}}\\
&\geq 
{{|\{c_1,\ldots,c_{\beta_p}\in \sC_p\mid \G_{\{c_1,\ldots,c_{\beta_p}\}}\mathrm{~contains~}\ZI^2\}|}\over {|\sC_p|^{\beta_p-\alpha_p}}}
\end{align*}
This shows that the random group in the $\alpha$-model contains $\ZI^2$ provided that the random in the $\beta$-model contains $\ZI^2$, for any $\beta$ dominating $\alpha$. This shows the existence of $\delta_{\ZI^2}$  and similarly of $\delta_{\neg \ZI^2}$.
\end{proof}

\begin{remark}
The above discussion of ``lift stable'' and ``factor stable'' properties also applies to the classical Gromov models and shows the existence of critical densities, for example in the density model. Recall that there exist morphisms between the random group of different models at the \emph{same} density, for example when comparing the usual model with the triangular model, see \cite[I.3.g]{ollivier}. The above results are of a different nature in that they concern morphisms between random groups at distinct density regimes in a given model. 
\end{remark}

 \section{General facts on the random group at positive density}\label{S - Phase transition}

The random group $\G$ in Gromov's density model exhibits the following well-known phase transition at density $\delta={1\over 2}$ (see \cite{gromov1} and Theorem 11  in \cite{ollivier}):
\begin{enumerate}
\item If $\delta<1/2$, then with overwhelming probability $\G$ is infinite hyperbolic, torsion free and of geometric dimension 2.
\item If $\delta>1/2$, then with overwhelming probability $\G$ is trivial (either $\{e\}$ or $\ZI/2\ZI$).
\end{enumerate} 
In other words, $\delta_{\Hyp}=\delta_{\mathrm{triv}}=\delta_{\neg \Hyp}=\delta_{\neg \mathrm{triv}}=1/2$.

The analogous result in our setting takes the following form. The proof, as in Gromov's case, is an illustration of the box principle.

\begin{proposition}\label{P- r separated} Let $(\G,X,\{\G_p\})$ be a deterministic data, and $r\in \IN_{\geq 2}$ be fixed. Then the random space $X_A$ in the density model over $(X,\G,\{\G_p\})$ exhibits a phase transition with overwhelming probability:
\begin{enumerate}
\item If $\delta<1/2$, then the random set $A\subset X$ is $r$-separated (that is, the simplicial distance between any two elements of $A$ is at least $r$).
\item If $\delta>1/2$, then there exists at least $r$  pairs of adjacent elements of $A$.  
\end{enumerate} 
\end{proposition}
\newcommand{\diam}{\mathrm{diam}}
\begin{proof}
 (1) Let $Y_1,Y_2,Y_3,\ldots$ be a sequence of independent uniformly distributed random variables with values in $\sC_p$
and for $m\leq i\leq |\sC_p|$ denote by $E_p^{(i)}(m)$ the event
\[
E_p^{(i)}(m)=\left \{ \diam(Y_{k_1},\ldots, Y_{k_m})\geq r\mid \forall~~1\leq  k_1<\cdots< k_m\leq i\right\}.
\]
We have 
\begin{align*}
\IP(E_p^{(i)}(m))&=\\
&\hskip-1.5cm\sum_{(c_1,\ldots,c_{i-1})\in \sC_p^{i-1}} \IP(E_p^{(i)}(m)\mid \{Y_1=c_1,\ldots, Y_{i-1}=c_{i-1}\}\cap E_p^{(i-1)}(m))\cdot\\
&\qquad\qquad\qquad\qquad\quad \IP(\{Y_1=c_1,\ldots, Y_{i-1}=c_{i-1}\}\cap E_p^{(i-1)}(m)).\\
\end{align*}
Given the event $\{Y_1=c_1,\ldots, Y_{i-1}=c_{i-1}\}\cap E_p^{(i-1)}(m)$, namely, that the faces $c_1,\ldots c_{i-1}$ have been chosen and any $m$ of them have diameter at least $r$, we obtain, letting $N_p(c_1,\ldots,c_{i-1})$ be the number of faces  in $X$ whose union with $m-1$ of the faces  $c_1,\ldots,c_{i-1}$ form a set of diameter at most $r$, 
\[
\IP(E_p^{(i)}(m)\mid \{Y_1=c_1,\ldots, Y_{i-1}=c_{i-1}\}\cap E_p^{(i-1)}(m))=1-{N_p(c_1,\ldots,c_{i-1})\over |\sC_p|}.
\]
Let $N_p$ be the the number of faces in the $r$ neighbourhood of a single chamber of $X$ (this doesn't depend on the given chamber). Then 
\[
N_p(c_1,\ldots,c_{i-1})\leq N_p \cdot |\{(k_1,\ldots, k_{m-1})\mid  1\leq k_1<\cdots< k_{m-1}\leq i-1\}|\leq N_p \cdot (i-1)^{m-1}
\] 
so that
\begin{align*}
\IP(E_p^{(i)}(m))&\geq\sum_{(c_1,\ldots,c_{i-1})\in \sC_p^{i-1}} \left (1-{N_p\cdot (i-1)^{m-1}\over |\sC_p|}\right) \cdot\\
&\qquad\qquad\qquad\qquad\qquad\qquad \IP(\{Y_1=c_1,\ldots, Y_{i-1}=c_{i-1}\}\cap E_p^{(i-1)}(m))\\
&=\left (1-{N_p\cdot (i-1)^{m-1}\over |\sC_p|}\right) \IP(E_p^{(i-1)}(m)). 
\end{align*}
Using $e^{-2x}\leq 1-x$ for all $x< 0.79$, we get
\[
\IP(E_n^{(i)})\geq e^{-2\sum_{j=1}^{i}{N_p\cdot (j-1)^{m-1}\over |\sC_p|}}\geq e^{-2N_p{i(i-1)\over |\sC_p|}}.
\] 
Therefore if $i=|\sC_p|^\delta$, where $\delta<1/2$, then $\IP(E_n)$ is arbitrarily close to 1. This concludes the proof of (1). The proof of (2) is of a similar nature. 
\end{proof}

Then the density model exhibits several extra phase transitions as the  parameter $\delta$ increases. The transitions depend on the local structure of the initial space $X$:

\begin{proposition}\label{d12} Let $(\G,X,\{\G_p\})$ be a deterministic data, and $r\in \IN_{\geq 2}$, $k,\ell\in \IN_{\geq 1}$ be fixed. Then the random space $X_A$ in the density model over $(\G,X,\{\G_p\})$ admits a phase transition at density $\delta={k\over {k+1}}$.  Namely, the following holds with overwhelming probability:
\begin{enumerate} 
\item If $\delta<{\frac {k} {k+1}}$, every ball $B$ of radius $r$ in $X^2$ contains to at most $k$ elements of $A$.
\item If $\delta> {\frac {k-1} {k}}$, there exists an $r$-separated set $E\subset X^{(1)}$ of cardinal at least $\ell$, such that every edge $e\in E$ is included in exactly $\min(k,q_{e}+1)$ elements of $A$.
\end{enumerate}
%In particular, if $q^*:=\max_{e\in X^{(1)}} q_e$ and $\delta> \frac{q^*}{q^*+1}$ then there exists there exists an $r$-separated set $E\subset X^{(1)}$ of cardinal at least $\ell$, such that every edge $e\in E$ is included in exactly $q_{e}+1$ elements of $A$.
\end{proposition}

\begin{proof}
(1) Let $Y_1,Y_2,Y_3,\ldots$ be a sequence of independent uniformly distributed random variables with values in $\sC_p$. For a ball $B\subset X^{(2)}$ of radius $r$, we write $Y_i\sim_B Y_j$ to mean that both $Y_i$ and $Y_j$ are included in  $B$.

Let $B\subset X^{(2)}$ be a ball of radius $r$ and fix $1\leq m_1<\cdots< m_{k}\leq i$. Since
\[
\IP(Y_{m_1}\sim_BY_{m_2}\sim_B\cdots \sim_B Y_{m_k})\leq \left( \frac{b^*}{|\sC_p|}\right)^k
\]
where $b^*:=\max_{B\subset X^{(2)}} |B_r|$ (which is finite as $X/\G$ is compact), we obtain
\begin{align*}
&\IP(\exists~~ 1\leq m_1<\cdots< m_{k}\leq i, ~~ \exists~~ B, ~~Y_{m_1}\sim_BY_{m_2}\sim_B\cdots \sim_B Y_{m_k})\\
&\hspace{1cm}\leq \sum_{1\leq m_1<\cdots< m_{k}\leq i} \sum_{f\in \sC_p}~~ \IP(Y_{m_1}\sim_{B_f}Y_{m_2}\sim_{B_f}\cdots \sim_{B_f} Y_{m_k})\\ 
&\hspace{1cm}\leq |\{(m_1,\ldots, m_{k})\mid  1\leq m_1<\cdots< m_{k}\leq i\}|\cdot  |\sC_p|\cdot \left( \frac{b^*}{|\sC_p|}\right)^k\\
&\hspace{1cm}\leq O\left( \frac{i^k}{|\sC_p|^{k-1}}\right)
\end{align*}
Therefore if $i=|\sC_p|^\delta$, where $\delta<k/(k-1)$, then with overwhelming probability every ball in $X^{(2)}$ of radius at most $r$ contains at most $k$ elements of the random subset $A$ of $\sC_p$, at density $\delta$.

(2) Let $Y_1,Y_2,Y_3,\ldots$ be a sequence of independent uniformly distributed random variables with values in $\sC_p$. Let $\sF_{p,k}$ be the set of all $k$-tuples of elements of $\sC_p$ whose faces are mutually adjacent to a same edge. For every $\psi\in \sF_{p,k}$ we write $E_{i,k}(\psi)$ for the event
\[
E_{i,k}(\psi) = \{ \exists~~ 1\leq m_1,\cdots, m_{k}\leq i ~(\forall r\neq s, ~m_r\neq m_s)\mid Y_{m_1}=\psi_{1},\ldots, Y_{m_k}=\psi_{k}\}
\]
and let $\chi_{i,k}(\psi)$ be the characteristic function of $E_{i,k}(\psi)$. We will estimate the expectation  of the random variable 
\[
Z_{p,k}:= \sum_{\psi\in \sF_{p,k}} \chi_{|\sC_p|^\delta,k} (\psi)
\]
whenever $\delta\in \left [{k-1\over k},{k\over k+1}\right]$. 
Fix $k$ and $\psi \in \sF_{p,k}$. Write 
\begin{align*}
\IP(E_{i,k}(\psi))=\IP\bigg( \bigcup_{m_1,\ldots,m_k=1\mid \forall r\neq s, m_r\neq m_s}^i Y_{m_1}=\psi_{1},\ldots, Y_{m_k}=\psi_{k}\bigg ).
\end{align*}
We have
\begin{align*}
\IP(E_{i,k}(\psi))&\geq\sum_{m_1,\ldots,m_k=1\mid \forall r\neq s, m_r\neq m_s}^i \IP(Y_{m_1}=\psi_{1},\ldots, Y_{m_k}=\psi_{k})\\
&\hspace{.5cm} - \sum_{m_1,\ldots,m_k=1\mid \forall r\neq s, m_r\neq m_s}^i\sum_{m_1',\ldots,m_k'=1\mid \forall r\neq s, m_r'\neq m_s'}^i\\
&\hspace{2.5cm} \IP(Y_{m_1}=\psi_{1},\ldots, Y_{m_k}=\psi_{k},Y_{m_1'}=\psi_{1},\ldots, Y_{m_k'}=\psi_{k})
\end{align*}
by inclusion-exclusion, where 
\[
\IP(Y_{m_1}=\psi_{1},\ldots, Y_{m_k}=\psi_{k})={1\over |\sC_p|^k}
\]
and 
\[
\IP(Y_{m_1}=\psi_{1},\ldots, Y_{m_k}=\psi_{k},Y_{m_1'}=\psi_{1},\ldots, Y_{m_k'}=\psi_{k})\leq {1\over |\sC_p|^{2k}}.
\]
Therefore,
\begin{align*}
\IP(E_{i,k}(\psi))&\geq  {i(i-1)\cdots(i-k-1)\over |\sC_p|^k}-{(i(i-1)\cdots(i-k-1))^2\over |\sC_p|^{2k}}.
\end{align*}
Substituting $|\sC_p|^\delta$ for $i$ we deduce
\begin{align*}
\IE(Z_{p,k})&= \sum_{\psi\in \sF_{p,k}}\IE(\chi_{|\sC_p|^\delta,k} (\psi))\\
&\geq |\sF_{p,k}| \bigg ( {|\sC_p|^\delta(|\sC_p|^\delta-1)\cdots(|\sC_p|^\delta-k-1)\over |\sC_p|^k}\\
&\hspace{3cm}-{{|\sC_p|^{2d}(|\sC_p|^\delta-1)^2\cdots(|\sC_p|^\delta-k-1)^2\over |\sC_p|^{2k}}}\bigg)
\end{align*}
Since for $p$ large enough,  
\[
{|\sC_p|^\delta-j\over |\sC_p|}<{1\over 2}
\]
whenever $\delta<1$ and $j\leq k+1$ we have
\[
{{|\sC_p|^{2d}(|\sC_p|^\delta-1)^2\cdots(|\sC_p|^\delta-k-1)^2\over |\sC_p|^{2k}}}<{|\sC_p|^\delta(|\sC_p|^\delta-1)\cdots(|\sC_p|^\delta-k-1)\over  2^k |\sC_p|^{k}}
\]
and therefore,
\begin{align*}
\IE(Z_{p,k})&\geq |\sF_{p,k}|\bigg (1-{1\over 2^k}\bigg) \bigg ( {|\sC_p|^\delta(|\sC_p|^\delta-1)\cdots(|\sC_p|^\delta-k-1)\over |\sC_p|^k}\bigg)\\
&\geq |\sF_{p,k}|\bigg (1-{1\over 2^k}\bigg) \bigg ({ (|\sC_p|^\delta-k-1)^k\over |\sC_p|^k}\bigg)
\end{align*}
Now, 
\[
|\sF_{p,k}| \geq {3|\sC_p|\over (q^*+1)}
\] 
so  
\begin{align*}
\IE(Z_{p,k})\geq {3\over (q^*+1)} \bigg(1-{1\over 2^k}\bigg) \bigg ({ (|\sC_p|^\delta-k-1)^k\over |\sC_p|^{k-1}}\bigg).
\end{align*}
showing that \[
\IE(Z_{p,k})\to\infty
\] 
whenever $\delta>{k-1\over k}$.  Since is $X$ uniformly locally bounded, this shows that we can find at least $\ell$ edges $e_1,\ldots, e_\ell$ in $X$ (for any $\ell$ fixed in advance), having at least $k$ chambers adjacent to $e_j$ removed.  In addition, since the number of such edges grows to infinity, we may suppose that the pairwise distance between these edges is at least $r$ from each other in $X$ (for any $r$ fixed in advance),  and  since the chambers are chosen independently at random, we may also suppose that exactly $\min(k,q_{e_j}+1)$ distinct chambers will be removed on $e_1,\ldots, e_\ell$.  
\end{proof}

\begin{proposition}\label{p-d1} Let $(\G,X,\{\G_p\})$ be a deterministic data and let $\ell\in \IN_{\geq 2}$. Then the random group $\G_A$ in the density model over $(X,\G,\{\G_p\})$ satisfies the following with overwhelming probability:
\begin{enumerate} 
\item If $\delta<\frac{q_*(X)}{q_*(X)+1}$, then $\G_A$ acts freely on a simply connected space $X$ without boundary with $X/\G_A$ compact.
\item If $\delta> \frac{q_*(X)}{q_*(X)+1}$, then $\G_A$ splits off a free factor isomorphic to a free group $F_\ell$ on $\ell$ generators:
\[
\G_A\simeq\Delta*F_\ell
\] 
where $\Delta$ is a finitely presented group, and $q_*(X):=\min_{e\in X^{(1)}} q_e$.
\end{enumerate}
In particular,
\[
\delta_{FA}((\G,X,\{\G_p\}))\leq  \frac{q_*(X)}{q_*(X)+1}<1
\] 
where $\delta_{FA}$ is the critical density for Serre's property FA (existence of fixed point for actions on trees). 
\end{proposition}

\begin{proof} The first assertion follows directly from Lemma \ref{d12} by contracting the free faces, so we prove (2). 
Since $X/\G$ is compact, there exists a constant $\gamma>0$ such that, if $q_*(X)$ denotes the lowest order of an edge in $X$, then 
\[
E_p:=\{e\in \sC_p^{(1)}\mid q_e=q_*(X)\}
\]
satisfies
\[
|E_p|\geq \gamma |\sC_p^{(1)}|
\]
for every $p>0$.  Let $k=q_*(X)$ and let $\tilde \sF_{p,k}$ be the set of all $k$-tuples of elements of $\sC_p$ whose faces are mutually adjacent to a same edge $e\in E_p$. The proof of Lemma \ref{d12}, with the same notation except for $\tilde \sF_{p,k}$ replacing $\sF_{p,k}$, shows that 
\begin{align*}
\IE(Z_{p,k})&\geq |\tilde \sF_{p,k}|\bigg (1-{1\over 2^k}\bigg) \bigg ({ (|\sC_p|^\delta-k-1)^k\over |\sC_p|^k}\bigg)
\end{align*}
and since now
\[
|\tilde \sF_{p,k}| \geq {3\gamma|\sC_p|\over (q^*+1)},
\] 
we obtain $\IE(Z_{p,k})\to\infty$. Therefore 
there exists, with overwhelming probability, an $r$-separated set $E\subset E_p$ of cardinal $\ell$, such that every edge $e\in E$ is included in exactly $q_{*}(X)+1$ elements of $A\subset \sC_p$. Let $V_{p}$ be the quotient space $X/{\G_p}$, and let $T_p$ be a maximal tree of the 1-skeleton of $V_p$, and choose a root $s_p\in T_p$. Let $V_p'$ be the topological space obtained from $V_p$ by collapsing $T_p$ to $s_p$, which is homotopy equivalent to $V_p$. 
The elements of $E/\G$, being face free edges in 
\[
V_A:=X_A/\G_p\subset V_p,
\] 
form a wedge sum $B_\ell$ of $\ell$ edges in the collapse $V_A'$ of $V_A$ in $V_p'$, and $V_A'$ is the wedge sum of $B_\ell$ and a compact complex $V_A''$. The van Kampen theorem shows that $\G_A\simeq\pi_1(V_A')$ splits as 
\[
\pi_1(V_A')=\pi_1(V_A'')*\pi_1(B_\ell)
\] 
where $\pi_1(B_\ell)\simeq F_\ell$ and $\pi_1(V_{A}'')$ is finitely presented.

In particular we have
\[
\delta_{\neg FA}((\G,X,\{\G_p\}))\leq  \frac{q_*(X)}{q_*(X)+1}
\] 
proving the proposition.
\end{proof}

\section{Random periodic flat plane problems II} \label{sect:density}

We now prove the implication
\[
\delta < \frac 5 8 \impl \ZI^2\inj \G
\]
in Theorem \ref{intrdensity} and similar assertions over local rings.  

In Theorem \ref{intrdensity} the fixed global field (denoted $k$) is $\FI_q(y)$ and the uniform lattice $\G$ in the deterministic data $(X,\G,\{\G_p\})$  is the Cartwright--Steger lattice \cite{CS} which is of type $\tilde A_2$ over $\FI_q((y))$. The presentation of $\G$ by Lubotzky, Samuels and Vishne in \cite{LSV} is especially useful, where the congruence subgroups $\{\G_p\}$  are described explicitly to illustrate the Ramanujan property of their quotients. We first review briefly the construction for notational purposes, referring to \cite{LSV} for complete details. 

 Let $D$ be the central simple algebra of degree 3 over $k$ defined by
\[
D=\bigoplus_{i,j=0}^2 k\xi_iz^j
\]  
with relations $z\xi_i=\phi(\xi_i)z$ and $z^3=1+y$, where $\phi$ is a generator of $\Gal(\FI_{q^3}/\FI_{q})$ and $\xi_i=\phi^i(\xi_0)$ is a basis for $\FI_{q^3}$ over $\FI_q$. Associated with $D$ are algebraic groups over $k$ defined by $\tilde \Gbf=D^\times$ and $\Gbf=D^\times/k^\times$. 
For a valuation $\nu$ on $k$ we let $D_\nu=D\otimes_k k_\nu$ and say that $D_\nu$  splits  whenever $D_\nu\simeq \mathrm{M}_3(k_\nu)$.

Let $T=\{\nu_{1/y}, \nu_{1+y}\}$ consisting of the degree valuation $\nu_{1/y}$ on $k$, and the valuation $\nu_{1+y}$ associated with the prime $1+y$, namely, $\nu_{1+y}((1+y)^if/g)=i$ where the polynomials $f,g$ are prime to $(1+y)$. Then by \cite[Prop. 3.1]{LSV} the algebra $D_\nu$ splits for all valuations $\nu\notin T$ on $k$, while it is a division algebra for $\nu\in T$. 

Let $\nu_0=\nu_y\notin T$ and $R_0=\{x\in k\mid \nu(x)\geq 0,\ \forall \nu\neq \nu_0\}=\FI_q[1/y]$. Since $D_{\nu_0}$ splits, we have $\Gbf(k_{\nu_0})\simeq \PGL_3(\FI_q((y)))$, so $\Gbf(R_0)$ embeds as a discrete subgroup of $\PGL_3(\FI_q((y)))$. Since $T\neq \emptyset$ and $\Gbf(k_{\nu})$ is compact for $\nu\in T$, then Strong Approximation  (cf.\ \cite[Section 4]{LSV}) shows that $\Gbf(R_0)$ is a cocompact lattice in $\PGL_3(\FI_q((y)))$. 

The Cartwright--Steger lattice \cite[Section 2]{CS} is the subgroup $\G$ of $\Gbf(R_0)$ defined as follows. We note that, as stated above, the lattice $\Gbf(R_0)$ is well defined only up to commensurability. A strict definition of $\Gbf(R_0)$ depends upon fixing a embedding $\Gbf$ into a linear group over $k$, which is chosen here to be $\GL_9(k)$  (see \cite[Prop. 3.3]{LSV}), so that $\Gbf(R_0)\egdef \Gbf(k)\cap \mathrm{M}_9(R_0)$.
Then $\G$ consists of matrices of $\Gbf(R_0)$ whose reduction modulo $1/y$ are upper triangular with $3\times 3$ identity blocks on the diagonal.

Another description of $\G$ from \cite[Section 4]{LSV} is as follows. Let $R$ be the subring of $k$ given by $R=\FI_q[y,1/y,1/(1+y)]$, and $A(R)$ be the $R$-algebra $A(R)=\bigoplus_{i,j=0}^2 R\xi_iz^j$ having the same defining relations as $D$, so that $D$ appears as the algebra of central fractions of $A(R)$, namely, $D=(R\backslash \{0\})^{-1}A(R)$.
We have $A^\times(R)/R^\times=\Gbf(R_0)$ (\cite[Prop. 4.9]{LSV}). Let 
\[
b=1-z^{-1}\in A^\times(R)
\]
and for $u\in \FI_{q^3}^\times\subset A^\times(R)$, set $b_u=ubu^{-1}\in A^\times(R)$. Let $\tilde \G$ be the subgroup of $A^\times(R)$ generated by the elements $b_u$. Then $\G$ is the quotient of $\tilde \G$ modulo $R^\times$. 

Write $X$ for the Bruhat--Tits building associated to $\PGL_3(\FI_q((y)))$. The vertices of $X$ are $\FI_q[[y]]$-lattices in $\FI_q((y))^3$ and incidence is given by flags. 
The action of $\Gbf(R_0)$ on $X$ (via its embedding in $\PGL_3(\FI_q((y)))$) is transitive on vertices.
The group $\G$ is a normal subgroup of $\Gbf(R_0)$ of finite index (see \cite[Theorem 2.6]{CS}) which also acts transitively on the vertices of $X$ (in fact, the reduced norm of $b$ is $y/(1+y)$, which coincide with  $y$ up to the invertible element $1+y$ of $\FI_q[[y]]$).
The congruence subgroups of $\G$ are defined as follows. Let $I$ be an ideal in $R$. The quotient map $R\to R/I$ induces an epimorphism $A(R)\to A(R/I)$, which itself induces a group morphism $A^\times (R)\to A^\times (R/I)$. Denote the kernel by $A^\times (R,I)$ and its quotient  modulo $R^\times$ by $(A^\times/R^\times) (R,I)$. Then we set $\G(I)=\G\cap  (A^\times/R^\times) (R,I)$, and let $\G_{I}=\G/\G(I)$ be the quotient group.

Let $f_p\in\FI_q[y]$ be a monic irreducible polynomial prime to $y$ and $y+1$,   $s_p\geq 1$ be an integer, and  $I_p=\langle f_p^{s_p}\rangle$. The deterministic data for Theorem \ref{intrdensity} is $(X, \G, \{\G(I_p)\})$ where:
\begin{itemize}
\item $X$ is the Bruhat-Tits building of $\PGL_3(\FI_q((y)))$
\item $\G$ is the Cartwright-Steger lattice in $\PGL_3(\FI_q((y)))$
\item $\{\G(I_p)\}$ are the arithmetic lattices associated with $I_p$. 
\end{itemize}
We suppose $\deg(f_p^{s_p})\to \infty$.

\newcommand{\M}{\mathrm{M}}

The assertion that $\ZI^2\inj \G$ for these random groups is related to the  elementary algebraic structure of $\PGL_3$ and $\PSL_3$ over fields and local rings.  As usual for a local ring $L$, $\GL_3(L)$ is the group of units of $\M_3(L)$, $\SL_3(L)$ is the subgroup of $\GL_3(L)$ of matrices with determinant 1, and $\PGL_3(L)$, $\PSL_3(L)$ are the projective versions.

We start with a lemma. Let $(L,m)$ be a finite local principal ring of prime characteristic with residue field $\FI_q=L/m$, uniformizer  $\pi$ and  length $s$ (the smallest integer $s>0$ such that $m^s=0$).

\begin{lemma}\label{lemdens1} 
The order of the group $\langle \gamma,\gamma'\rangle$ generated by two commuting elements  $\gamma,\gamma'\in \PGL_3(L)$ is at most
\[
3{q^{2\lceil\log_qs\rceil}(q^3-1)}
\] 
where $\log_q$ is the logarithm relative to $q$, and $\lceil\cdot\rceil$ is the upper integer value. 
\end{lemma}

\begin{proof} 
In the reduction modulo $m$
\[
\PGL_3(L)\ni \gamma\mapsto \bar \gamma\in \PGL_3(\FI_q),
\] 
an element $\gamma\in \PGL_3(L)$ such that ${\bar \gamma}=1$ has a representative of the form 
\[
\lambda(1+\pi^k \gamma_0)\in \GL_3(L)
\] 
for some $\lambda\in L^\times$, $\gamma_0\in \M_3(L)$ and  $k\geq 1$. The Frobenius map gives
\[
(1+\pi^k\gamma_0)^{q^{\lceil\log_qs\rceil}}=1
\]
therefore $\gamma^{q^{\lceil\log_qs\rceil}}=1$. 

Thus, if $\gamma$ and $\gamma'$ are two commuting elements in $\PGL_3(L)$, then the kernel $K$ in the exact sequence
\[
1\to K\to  \langle \gamma,\gamma'\rangle\to \langle \bar \gamma,\bar \gamma'\rangle\to 1
\]
is an abelian group of order at most $q^{2\lceil\log_qs\rceil}$.

We now estimate $|\langle \bar \gamma,\bar \gamma'\rangle|$. Assume without loss of generality  that $L=\FI_q$.  Let $\underline{\gamma},\underline{\gamma}'$ be representatives of $\gamma,\gamma'$ in $\GL_3(L)$, and   
let $f\in L[X]$ be the characteristic  polynomial of $\underline\gamma$. 

\newcommand{\ord}{\mathrm{ord}}

Assume that $f$ doesn't have a root in $L$. Then $f$ is irreducible over $L$ (since a factorization would split off a degree one factor) and therefore 
equal to its minimal polynomial. It follows that the rational canonical form of $\underline \gamma$ has only one block associated with $f$, and $\underline \gamma$ has a cyclic vector $\xi$. Given $\gamma'\in \PGL_3(L)$,  there exists $P\in L[X]$ such that 
\[
\underline \gamma'\xi=P(\underline \gamma)\xi.
\] 
If  
\[
\underline \gamma\,\underline \gamma'=\lambda\underline  \gamma'\underline \gamma
\]
 for some  $\lambda\in L^\times$, then
  \[
  \underline \gamma^k\underline \gamma'=\lambda^k\underline  \gamma'\underline \gamma^k
  \] 
  so
\[
\underline \gamma'\underline\gamma^k\xi=\lambda^{-k}\underline\gamma^kP(\underline\gamma)\xi=\lambda^{-k}P(\underline\gamma)\underline\gamma^k\xi.
\]
Thus, 
\[
\underline\gamma'=P(\underline\gamma)\diag{1}{\lambda^{-1}}{\lambda^{-2}}
\]
in the basis $(\xi,\underline \gamma\xi,\underline \gamma^2\xi)$. Taking determinants the equality $\underline \gamma\,\underline \gamma'=\lambda\underline  \gamma'\underline \gamma$ shows that $\lambda$ is a cube root of 1. Thus every element of $\langle \underline \gamma,\underline \gamma'\rangle$ belong to the subspace 
\[
\left\{P(\underline\gamma)\mid P\in L[X]\right\}
\]
 of $M_3(L)$, possibly after  multiplying by an element of the form 
 \[
 \diag{1}{\lambda}{\lambda^2}
 \] 
 $\lambda\in L^\times$ a cube root of 1. Since the dimension of $\left\{P(\underline\gamma)\mid P\in L[X]\right\}$ over $L$ is 3, the order of $\langle \underline \gamma,\underline \gamma'\rangle$ is at most $3(q^3-1)$. 

A similar argument works more generally if $f$ equals the minimal polynomial. Otherwise, $f$ has a root $a$ in $L$ and  we may assume up to conjugacy that $\underline \gamma$ has the form 
$\smatr{a}{0}{0}{\alpha}$
where $a\in L^{\times}$ and $\alpha\in \GL_2(L)$. Write 
\[
\underline\gamma'=\smatr{b}{m}{n}{\beta},
\] 
$b\in L^{\times}$, $m^t,n\in L^2$, and $\beta\in \GL_2(L)$. Since 
\[
\underline\gamma\underline\gamma'=\lambda\underline\gamma'\underline\gamma
\]
 for some $\lambda\in L^\times$, we have  $ab=\lambda ba$
  so 
  \[
  \lambda=1
  \] and thus $m\alpha =am$, $\alpha n = a n$ and 
  \[
  \alpha\beta=\beta\alpha.
  \]
  If $m=n=0$, then the argument of the previous \textsection\  shows that the order of $\langle \alpha,\beta\rangle$ is at most $q^2-1$, so the order of $\langle  \gamma, \gamma'\rangle$ is at most $q^3-1$. Otherwise, $a$ is an eigenvalue of $\alpha$ and it is easy to see that the order of $\langle  \gamma, \gamma'\rangle$ is at most $q^3-1$ in this case too, which proves the lemma. 
\end{proof}

\begin{proof}[Proof of the assertion that $\ZI^2\inj \G$ in Theorem \ref{intrdensity}]

The local ring 
\[
R/I_p\simeq \FI_q[y]/\langle f_p(y)^{s_p} \rangle
\]
has residue field  
\[
\F_{q_p}=R/\langle f_p\rangle\simeq \FI_q[y]/\langle f_p(y)\rangle.
\]
By Theorem 6.6 in \cite{LSV}
\[
\PSL_3(R/I_p)\subset \G_{I_p}\subset \PGL_3(R/I_p).
\]
Let $\Lambda$ be any subgroup of $\G$ isomorphic to $\ZI^2$. 
Lemma \ref{lemdens1} shows that the order of the image of $\Lambda$ in   $\G_{I_p}$ is at most 
\[
3{q_p^{2\lceil\log_{q_p}s\rceil}(q_p^3-1)}
\] 
where $q_p$ is the order of the residue field. 
Denote by $\Pi\simeq \RI^2$ the flat associated to $\Lambda$ in $X$ and  let $\sC_p$ be the set of $\G(I_p)$-orbits of chambers in $X$.
 The set $F_p$ of  $\Lambda\cap \G(I_p)$ orbits of chambers in $\Pi$ is finite, and there is  a constant $C$ (namely, $C=3\times$the number of $\Lambda$ orbits of chambers in $\Pi$) such that 
\[
|F_p|\leq Cq_p^{2{\lceil\log_{q_p}s_p\rceil}+3}.
\] 
Let $F_p'$ be the image of $F_p$ in $\sC_p$.

Consider a sequence $Y_1,Y_2,Y_3,\ldots$ of independent identically distributed random variable with values in $\sC_p$.
The probability $\IP(E_n)$ of the event:

\[
E_n= \{Y_i\notin F_p', \; \forall i\leq |\sC_p|^\delta\}
\]
can be estimated below by 
\[
\IP(E_n)\geq \left(1-{|F_p'|\over |\sC_p|}\right)^{|\sC_p|^\delta}
\]
therefore 
\[
\IP(E_n)\geq e^{-2|\sC_p|^{\delta-1}|F_p'|}.
\] 

\noindent Since $\G_{I_p}$ acts freely on $\sC_p$ with $\sC_p/\G_{I_p}$ fixed,
 \[
|\sC_p|\geq D|\G_{I_p}|
\]
for some constant $D>0$. Let us compute $|\G_{I_p}|$. The reduction map 
\[
R/I_p\to R/\langle f_p\rangle=\FI_{q_p}
\]
induces a surjective map
\[
\pi: \GL_3(R/I_p)\to \GL_3(\FI_{q_p})
\]
since $\det \colon \GL_3(R/I_p)\to (R/I_p)^\times$
commutes with reduction modulo $\langle f_p\rangle$, and since an element is invertible in $R/I_p$ if and only if its image in $\FI_{q_p}$ is nonzero  (as $R/I_p$ is local). Therefore any pull-back of a matrix of $\GL_3(\FI_{q_p})$ in $\M_3(R/I_p)$ is a matrix in   $\GL_3(R/I_p)$.
The kernel of $\pi$ consists of matrices of the form $\id+\gamma_0$ where all coefficient of $\gamma_0$ belong to the ideal $\langle f_p\rangle$.
Therefore 
\begin{align*}
\left | \GL_3(R/I_p)\right |&=q_p^{9(s_p-1)} \left | \GL_3(\FI_{q_p})\right |= q_p^{9s_p-6}(q_p^3-1)(q_p^2-1)(q_p-1)
\end{align*}
%where $\left | \PGL_3(R/I_p)\right |= q_p^{9s_p-6}(q_p^3-1)(q_p^2-1)(q_p-1)/(q_p-1)q_p^{s_p-1}=q_p^{8s_p-5}(q_p^3-1)(q_p^2-1)$. 
As $\det$ 
is surjective and $a\in (R/I_p)^\times$ if and only if $a$ is not a multiple of $f_p$, we have that
\begin{align*}
|\SL_3(R/I_p)| & ={1 \over |(R/I_p)^\times|}|\GL_3(R/I_p)|\\
&={1 \over q_p^{s_p-1}(q_p-1)}q_p^{9s_p-6}(q_p^3-1)(q_p^2-1)(q_p-1)\\
&=q_p^{8s_p-5}(q_p^3-1)(q_p^2-1).
\end{align*}
This gives
\begin{align*}
|\G_{I_p}|\geq |\PSL_3(R/I_p)| &=\mu_p^{-1} |\SL_3(R/I_p)|=\mu_p^{-1}q_p^{8s_p-5}(q_p^3-1)(q_p^2-1)
\end{align*}
where  $\mu_p=|\{a\in(R/I_p)^\times\mid a^3=1\}|$.
Therefore,
\begin{align*}
|\sC_{n}|^{\delta-1} |F_p'|&\leq CD^{\delta-1}\mu_p^{1-\delta}q_p^{2{\lceil\log_{q_p}s_p\rceil}+3}q_p^{8s_p(\delta-1)}\\
&=CD^{\delta-1}\mu_p^{1-\delta}q_p^{2{\lceil\log_{q_p}s_p\rceil}+8s_p(\delta-1)+3}.
\end{align*}
Thus, since ${\lceil\log_{q_p}s_p\rceil\over s_p}\longrightarrow_{n\to \infty} 0$ when $\deg f_p^{s_p}\to \infty$, we obtain:

\begin{itemize}
\item[(i)] If  $s_p=1$ for large $n$, then $(R/I_p)^\times$ is cyclic, $\mu_p\leq 3$ (for large $n$) and thus 
$\IP(E_n)\to1$ 
if
\[
8(\delta-1)+3<0
\]
that is,  
\[
\delta<{5\over 8}
\]

\item[(ii)] if $s_p\geq k\geq 2$ for large $n$, then $(R/I_p)^\times$ is not cyclic in general, but using the rough estimate 
$\mu_p\leq q_p^{s_p-1}$ (or $\mu_p\leq q_p-1$ if the  order satisfy $q\neq 3^\ell$,  $\ell=1,2,3,\ldots$), we have that $\IP(E_n)\to1$ 
if  
\[
(1-\delta)(s_p-1)+ 8s_p(\delta-1)+3<0
\]
that is
\[
\delta< {7k-2\over 7k+1}
\]

\item[(iii)] if $(f_p^{s_p})$ doesn't satisfy (i) or (ii), we separate the cases $s_p\geq 2$ and $s_p=1$ and obtain the desired conclusion when $\delta<{5\over 8}$ as  ${7k-2\over 7k+1}>{5\over 8}$ for $k\geq 2$. 

\end{itemize}

This shows that $\IP(E_n)\to1$. Therefore,  Lemma \ref{L - lifting lemma} applies with overwhelming probability and we have that $\ZI^2\inj \G$ with overwhelming probability.
\end{proof}

\begin{remark}
We have proved the stronger result that every embedding $\ZI^2\inj \G_3$ eventually lifts to an embedding $\ZI^2\inj \G$ with overwhelming probability. Recall that according to Mostow and Prasad, the set of embeddings $\ZI^2\inj \G_3$ is dense \cite{mostow}.  
\end{remark}

\begin{question} Are these random groups residually finite?   (For comparison, note that there exist non residually finite \cite{Deligne} central extensions of the form $1\to \ZI\to\widetilde{ \Sp({2n},\ZI)}\to \Sp({2n},\ZI)\to 1$.)
\end{question}

\begin{question} What is the value of $\delta_{\ZI^2}(X_3,\G_3,\{\G_3(I_p)\})$? What is the value of   $\delta_{\IR^2}(X_3,\G_3,\{\G_3(I_p)\})$?
\end{question}

\begin{question}[``Perforated buildings''] If $X$ is a Bruhat--Tits building of rank 2 (possibly exotic) and $Y\subset X$ is an $R$-separated subset of chambers, does $\IR^2\inj X\setminus Y$ for $R$ large enough? 
\end{question}

\begin{remark}
It is proved in \cite{LSV}  that the congruence quotients of $X_3$ by $\G_3(I_p)$ are Ramanujan complexes. This fact,  whose proof relies on the ``Jacquet-Langlands correspondence" (in positive characteristic), is not used in the present paper. (Nevertheless, the random quotients $X/\G$ as in Theorem \ref{intrdensity}  are Ramanujan complexes ``up to $\e$", at least if the density is not too large.) 
The main input from algebraic groups used in the proof is Strong Approximation, and so it is most likely that results in the spirit of Theorem \ref{intrdensity} can also be established for other algebraic groups (e.g. $\Sp$) and over other fields (e.g.\ characteristic 0), the first step being to construct explicit families of lattices.  We also note that the case of positive characteristic is easier to handle from an algorithmic perspective (the construction of $f_p$ can be done in polynomial time using the Shoup algorithm).
\end{remark}

\section{Further properties of the random group}\label{propertyT}

The critical density $\delta_T(X,\G,\{G_p\})$ for Kazhdan's property T can be estimated using the so-called ``spectral criterion for property T", which is local. This is done by expliciting (spectral) bounds on the local geometric perturbations, following \cite[\textsection 6]{Garland} and \cite{FH}. 

We recall (see \cite{HV} for more details) that the ``spectral criterion" refers to the following well--known result of Garland \cite[Theorem 5.9]{Garland}:

\begin{theorem}[Garland]\label{T-Garland}
Let $V$ be finite simplicial complex of dimension $\ell$ in which every simplex belongs to an $\ell$-cell, and let $r\leq \ell$. Assume that for every vertex $v\in V^{(0)}$, the $r$-th reduced cohomology group $\tilde H^r(\Lk v,\RI)$ [i.e.\ $\tilde H^r=H^r$ if $r\geq 1$ and $\tilde H^0=\coker (H^0(\star)\to H^0)$] of the link $\Lk v$ of $V$ vanishes, and denote by $\lambda_r(\Lk v)$ the smallest nonzero eigenvalue of $\Delta^+=d\del$ acting on $C^r(\Lk v)$. Set $\lambda_r:=\min_v \lambda_r(\Lk v)$. 
If 
\[
\lambda_{r-1}\geq {\ell-r\over r+1}
\] 
then $H^r(V,\rho)=0$ for every finite dimensional unitary representation $\rho$ of $\pi_1(V)$.  

(If $\ell=2$, $r=1$, this says
 \[
 \lambda_0>1/2\impl H^1(V,\rho)=0
 \]
 for any $\rho$.)
\end{theorem}

The statement in \cite{Garland} deals only with those $V$ whose universal coverings are Bruhat-Tits buildings, but Borel points out  \cite[Theorem 1.5]{Borel} that Garland's argument is  general.  (While we are mostly interested in Bruhat-Tits buildings here, it is the general case that is useful for us.) Borel also notes in \textsection 2.2 that the assumption that $V$ is finite can be replaced by ``uniformly locally finite", provided that the conclusion refers to the vanishing of $L^2H^r(V,\rho)$ ($L^2$-summable cochains). It was later shown that the family of admissible $\rho$  with vanishing $H^r$ could be substantially enlarged, in particular, the linearity assumption on $\rho$ can be disposed of (in the smooth category \cite{Wang}), as can the finite dimensionality assumption on $\rho$ (see \cite{Pansu,  Zuk, BS}).  The latter leads to the conclusion that  $H^r(V,\rho)=0$ for every unitary representation $\rho$ of $\pi_1(V)$ which, in the case $r=1$, is equivalent to $\G$ having Kazhdan's property T by a well--known result of Delorme and Guichardet (see \cite{HV}). Garland studies in particular the 2-dimensional case, exploiting the computation of the exact value of $\lambda_0$ for Tits buildings  (associated with BN-pairs, the computation appears in \cite{FH}  and \cite[Proposition 7.10] {Garland}) to derive the vanishing of $H^1$ ``when the order is large enough". (As Garland observes, his cohomology vanishing results for lattices -- with finite dimensional unitary targets -- were in fact already known in degree 1 as they were covered by Kazhdan's work.)

We now explain why property T holds for random groups ``when the order is large enough" 
(in particular when $q\geq 5$ if $\delta<1/2$, and  the deterministic data is of type $\tilde A_2$, and for $q\gg1$ if  $\delta$ is only assumed to be bounded away from 1).

The idea is that $\lambda_0$ is controlled if the number of chambers removed is significantly smaller than the order  (defined as the minimal number of chambers on an edge minus 1) of $V$:

\begin{lemma}\label{L-large order} For every $\delta>0$ and every integer $n\geq0 $, there exists a constant $\delta'>0$ such that if $G$ is a graph of order $q\geq 5n/\delta$ with $\lambda_0(G)>1-\delta'$, and $G'\subset G$ is a subgraph with $|{G'}^{(1)}|\geq |G^{(1)}|-n$, then $\lambda_0(G')\geq 1- \delta$.
\end{lemma}

Indeed, by the Cheeger inequality (for graphs), it is enough to prove the corresponding statement for the Cheeger constant, then the numerator of the two isoperimetric ratios differ by  a constant  that is washed out as $q\to \infty$. 

More explicitly:

\begin{proof}
For a subset $A\subset V$, let $h_G(A):={|\del_G A|\over \min(|A|_G,|X-A|_G)}$, where $|\del_G A|$ is the number of  edges with extremities belonging to both $A$ and $X-A$, and $|A|_G:=\sum_{x\in A}\val_G(x)$. The Cheeger constant of $G$ is $h(G):=\min_{A\subset G} h_G(A)$.
Let $A$ be a subset of $V'$ such that $h(G')=h(A)$. 
%Then, if $R\subset A$ denotes the set of vertices containing an edges in $G\setminus G'$, we have
%\begin{align*}
%|A|_{G'}&=\sum_{v\in A} \val_{G'} (A)\\
%&=\sum_{v\in A-R} \val_{G'} (A)+\sum_{v\in R} \val_{G'} (A)\\
%&\geq \sum_{v\in A-R} \val_{G} (A)+\sum_{v\in R} (\val_{G} (A)-n)\\
%&=|A|_{G}-n|R|\\
%&\geq |A|_G-2n^2
%\end{align*}
%and similarly for $|X-A|_{G'}$. 
By the Cheeger inequality,
\begin{align*}
\lambda_0(G)&\leq 2h(G)\leq {2|\del_G A|\over \min(|A|_G,|X-A|_G)}\\
&\hskip1cm\leq {2|\del_{G'} A|+2n\over \min(|A|_{G'},|X-A|_{G'})}=2h(G')+{2n\over \min(|A|_{G'},|X-A|_{G'})}
\end{align*}
Since $q\geq 5n/\delta$, we obtain $\lambda_0(G)\leq 2h(G')+\delta/2$. Thus
\[
\lambda_0(G)\leq 2\sqrt{2\lambda_0(G')} +\delta/2
\] 
using again the Cheeger inequality.
\end{proof}

Let $X$ be a thick  irreducible Euclidean building of rank 2. The links of $X$ are spherical buildings and there exist  integers $q^*\geq q_*\geq 2$ such that the degree of an edge in $X$ is either $q^*+1$ or $q_*+1$. We call $q_*$ the order of $X$. In the $\tilde A_2$ case,  $q^*=q_*$.

\begin{theorem}[Generic property T]\label{T - T}
For every $\delta_0<1$, there exists $q_0$ such that if $\delta<\delta_0$ and if in the deterministic data $(X,\G,\{\G_p\})$, the space  $X$ is a thick classical (i.e.\ associated with an algebraic group over a local field) irreducible Bruhat--Tits building of rank 2 and order $q_*\geq q_0$,  then the random group at density $\delta$ with deterministic data $(X,\G,\{\G_p\})$ has Kazhdan's property T with overwhelming probability.  Thus 
\[
\delta_T(X,\G,\{\G_p\})\to 1
\] 
as $q_*(X)\to \infty$.
\end{theorem}

\begin{remark}
The groups appearing in the above result (and in Corollary  \ref{C- random building with chambers missing bounded model}) are buildings with chambers missing in the sense of \cite{chambers}. Some (weak) buildings with chambers missing have the ``opposite" Haagerup property, see \cite{chambers}.  We also note that  random groups in the density model  admit an infinite quotient with property T, \emph{independently of the order}, as soon as we are given a  deterministic data with property T.  
\end{remark}

\begin{proof}
Let $k$ be large enough so that $\delta_0<{k\over k+1}$.  Let $0<\e<1/2$ be given. By Lemma \ref{L-large order}, we can find $\e'>0$ such that if $\lambda_0(G)>1-\e'$ then $\lambda_0(G-k)>1-\e>1/2$. Feit--Higman \cite{FH}  and Garland \cite{Garland} show that if $G_q$ is a spherical building of order $q$, then  $\lambda_0(G_q)$ converges to 1 as $q\to \infty$.  Let $q_0$ be larger than $10k$ and be such that $\lambda_0(G_q)>1-\e'$ for any spherical building $G_1$ of order $q\geq q_0$.  If in the deterministic data $(X,\G,\{\G_p\})$, the space  $X$ is a thick classical irreducible building of rank 2 such that $q_*(X)\geq q_0$, and if $\delta<\delta_0$, then by Proposition \ref{d12}, the links of $X$ contain at most $k$ chambers missing at every vertex, with overwhelming probability.  Therefore, the random perforated building $X\setminus \G A$ at density $\delta$ satisfies $\lambda_0(X\setminus \G A)\geq 1-\e>1/2$ with overwhelming probability.
Since the random universal cover $\widetilde{X\setminus \G A}$ is locally isomorphic to  ${X\setminus \G A}$, Theorem \ref{T-Garland} applies. Thus the random group has property T (and the spectral gap is arbitrarily close to 1) with overwhelming probability.      
\end{proof}

It is natural to wonder if the bound on the order can be estimated. It seems difficult to give a general answer, but we can do it at least when the density is $<1/2$ (or in the model with few chambers missing). In fact, one can compute the \emph{exact} value of $\lambda_0(X)$, in the spirit of  Feit and Higman's paper \cite{FH}.

More precisely: 

\begin{theorem}
If $(X,\G,\{\G_p\})$ is of type $\tilde A_2$, then both in the bounded model and in the density model of parameter $\delta<1/2$,  the random group has property T (with overwhelming probability) when $q\geq 5$.  
\end{theorem}

This follows from:

\begin{proposition}\label{t-1}
Let $G$ be a spherical building of type $A_2$ and order $q$ with a single chamber missing. Then
\[
\lambda_0(G)=1-{\sqrt{q+{1/4}}+ {1/2}\over q+1}.
\]   
In particular, $\lambda_0(G)>1/2$ whenever $q\geq 5$.
\end{proposition}

We recall (see e.g.\ \cite{HV}) that the value computed by Feit and Higman in the $A_2$ case is 
\[
\lambda_0(G)=1-{\sqrt{q}\over q+1}
\]
so in particular, $\lambda_0(G)>1/2$ whenever $q\geq 2$.

\begin{proof}
Let $G\leadsto G'$ be an extension into an $A_2$ building of order $q$ (see \cite{chambers}) and let $P$ be the  projective plane corresponding to $G'$. We choose the following basis for the space of functions on the vertex set of $G$:
\[
(\delta_{p_1},\delta_{p_2},\ldots \delta_{p_{q^2+q+1}},\delta_{l_1},\delta_{l_2},\ldots,\delta_{l_{q^2+q+1}})
\]
where $p_1\in G$ (resp. $l_1\in G$) corresponds to the point $p$ (resp. the line $l$) of $P$ associated to the missing chamber, while $p_2,\ldots,p_{q+1}$ (resp. $l_2,\ldots,l_{q+1}$) correspond to an enumeration of the points of $l$ distinct from $p$ (resp. the lines adjacent to $p$ distinct from $l$) in $P$.

The Laplace operator $\Delta$ is of the form

\[
\Delta=\Id-\left(
\begin{array}{cc}
0&A\\
A^t&0
\end{array}  
\right),
\]
where $A$ is the normalized adjacency matrix, namely $A=D_q^{-1/2}A_0D_q^{-1/2}$, where $A_0$ is the usual (bipartite) adjacency matrix (we recall that $(A_0)_{i,j}=1$ when $l_i$ is adjacent to $p_j$), and $D_q$ is the diagonal matrix having $q+1$ down the diagonal, except on the first entry which is  $q$. Denote $\tilde q=\sqrt{q(q+1)}$. A computation shows that
\[
AA^t=A^tA=\left(
\begin{array}{ccc}
q\tilde q^{-2}&0_q&(q+1)^{-1}\tilde q^{-1}\mathbf{1}_{1\times q^2}\\
0_q&B_q&(q+1)^{-2}\mathbf{1}_{q\times q^2}\\
(q+1)^{-1}\tilde q^{-1}\mathbf{1}_{q^2\times 1}&(q+1)^{-2}\mathbf{1}_{q^2\times q}&C_{q^2}\\
\end{array}  
\right)
\] 
where  $B_q$ is the $q\times q$ matrix with 
\[
{q\over (q+1)^{2}}+{1\over {q(q+1)}}
\] 
on the diagonal, and ${q-1}\over {q(q+1)}$ elsewhere, while $C_{q^2}$  is the $q^2\times q^2$ matrix with $(q+1)^{-1}$ on the diagonal and $(q+1)^{-2}$ elsewhere. Set 
\[
\gamma_q=(q+1)\tilde q^{-1}=\sqrt{1+{1\over q}}.
\] 
We have
\[
(q+1)^2A^tA-q\id =\left(
\begin{array}{ccc}
1& 0_q & \gamma_q\mathbf{1}_{1\times q^2}\\
0_q&\gamma_q^2\mathbf{1}_{q\times q}&\mathbf{1}_{q\times q^2}\\
\gamma_q\mathbf{1}_{q^2\times 1}&\mathbf{1}_{q^2\times q}&\mathbf{1}_{q^2\times q^2}
\end{array}  
\right)
\]
whose kernel is of codimension 3. Consider the vectors
\[
\begin{array}{ccc}
v_1=\left(\begin{array}{c}1\\0_q\\ \gamma_q\mathbf{1}_{q^2\times 1}\end{array}\right)&
v_2=\left(\begin{array}{c}0\\ \gamma_q^2\mathbf{1}_{q\times 1}\\\mathbf{1}_{q^2\times 1}\end{array}\right)&
v_3=\left(\begin{array}{c}\gamma_q\\ \mathbf{1}_{q^2+q\times 1}\end{array}\right).
\end{array}
\]
A direct computation shows that
\begin{align*}
(q+1)^{2}A^tAv_1&=(q+1)v_1+q^2\gamma_qv_3\\
(q+1)^{2}A^tAv_2&=(q+q\gamma_q^2)v_2+q^2v_3\\
(q+1)^{2}A^tAv_3&=\gamma_qv_1+qv_2+(q^2+q)v_3
\end{align*}
Thus the problem reduces to computing the spectrum of the $3\times 3$ matrix
\[
\left(\begin{array}{ccc}
q+1&0&\gamma_q\\
0&2q+1&q\\
q^2\gamma_q&q^2&q^2+q
\end{array}\right)
\]
The characteristic polynomial is 
\[
-x^3+(q^2+4 q+2) x^2-(2 q^3+6 q^2+4 q+1) x+q^2 (q+1)^2
\]
where $(q+1)^2$ is an obvious root. The two other roots are
\[
x = (2 q+1-\sqrt{4 q+1})/2
\]
and 
\[
x = (2 q+1+\sqrt{4 q+1})/2.
\]
Therefore, the eigenvalue of $\left(
\begin{array}{cc}
0&A\\
A^t&0
\end{array}  
\right)$ we are interested in 
is 
\[
{1\over q+1} \sqrt{q+{1\over 2}+\sqrt{q+{1\over 4}}}={{1\over 2}+\sqrt{q+{1\over 4}} \over q+1}.
\]
 \end{proof}

A similar computation can be worked out when two chambers are missing (arguing according to the respective positions of the chambers in an apartment)  but the value of 
\[
\lambda_{0, k}(G):=\inf_{E\subset G^{(1)},\,\, |E|=k}\lambda_0(G\setminus E)
\]
when $G$ is a spherical building (of dimension 1) might be more challenging to find. It would also be interesting to make similar estimates with more general (e.g.\ smooth) targets $\lambda_{0, k}(G,Y)$ and deduce the corresponding fixed points theorem for the random groups following \cite{Wang, gromov01}.

The results in \cite{chambers} on buildings with chambers missing and \cite{rd}  on property RD can be applied to our  random groups in the $\tilde A_2$ case. This shows in particular that the random group is \emph{rigid} in the sense that it remembers the building that it comes from (see \cite[Section 5]{chambers}). Furthermore,   by Corollary 6 in \cite{rd}, the random group satisfies the Baum--Connes conjecture without coefficients.

\begin{theorem} If the deterministic data is of type $\tilde A_2$, then

\begin{enumerate}
\item the random group  $(\G,X)$ at density $\delta<\frac 1 2$ admits unique extension  $(\G,X)\leadsto (\G',X')$ (in the sense of \cite[Section 1]{chambers}) into a Euclidean building.
\item the random group  $\G$ at arbitrary density satisfy the Baum--Connes conjecture without coefficients.  
\end{enumerate}
\end{theorem}

The first assertion is a   consequence of Proposition \ref{P- r separated} and \cite[Theorem 5.11]{chambers}.
It is unknown if $\G$ satisfies the Baum--Connes conjecture with coefficients. The conjecture without coefficients is also unknown for the other Coxeter types (in the irreducible case) but would follow  from the ``interpolation of property RD  to intermediate rank situations'' (see \cite{rd}), that is, between hyperbolic groups and (uniform) higher rank lattices of the corresponding type, provided that the latter groups satisfy property RD (which is conjectured by  A.\ Valette).

\end{document}